\documentclass[letterpaper, 10pt, journal, onecolumn, final]{IEEEtran}
\usepackage{cite}
\usepackage{algorithm}
\usepackage{algpseudocode}
\usepackage{algorithmicx}
\usepackage{graphicx}
\usepackage{textcomp}
\usepackage{amssymb,amsmath,amsfonts,cases,mathtools}
\usepackage{microtype}
\usepackage{url}
\usepackage{hyperref}
\usepackage{xifthen}
\usepackage{xpatch}
\usepackage{amsthm}
\usepackage[dvipsnames]{xcolor}
\usepackage{tikz}
\usepackage{subfig}
\usepackage{lipsum,booktabs}
\usepackage{pifont}
\usepackage{xcolor}
\usepackage{soul}
\usepackage{balance}

\hypersetup{hidelinks}

\allowdisplaybreaks

\newcommand{\n}{n}

\newcommand{\cD}{\mathcal{D}}

\newcommand{\cF}{\mathcal{F}}

\newcommand{\cR}{\mathcal{R}}
\newcommand{\cS}{\mathcal{S}}

\newcommand{\R}{\mathbb{R}}
\newcommand{\N}{\mathbb{N}}

\newcommand{\x}{x}

\newcommand{\z}{z}

\newcommand{\X}{X}

\newcommand{\pr}{\delta}

\newcommand{\iter}{{t}}
\newcommand{\iterp}{{\iter + 1}}

\newcommand{\ud}{_}

\newcommand{\xt}{\x\ud{\iter}}

\newcommand{\zt}{\z\ud{\iter}}

\newcommand{\xtp}{\x\ud{\iterp}}

\newcommand{\ztp}{\z\ud{\iterp}}

\newcommand{\xstar}{\x\ud\star}

\newcommand{\ut}{u\ud{\iter}}

\newcommand{\utp}{u\ud{\iter+1}}

\newcommand{\bx}{\bar{\x}}
\newcommand{\bz}{\bar{\z}}

\newcommand{\str}{s}

\newcommand{\zeq}{\z\ud{\text{\tiny FIX}}}

\newcommand{\norm}[1]{\left \|#1 \right \|}

\newcommand{\distsimbol}{d}
\newcommand{\dist}[2]{\distsimbol\left (#1,#2 \right )}

\newcommand{\normf}[1]{\norm{#1}_{F}}
\newcommand{\norms}[1]{\norm{#1}_{S}}

\newcommand{\distfs}{\distsimbol_F}
\newcommand{\distss}{\distsimbol_S}

\newcommand{\distf}[2]{\distfs\left (#1,#2 \right )}
\newcommand{\dists}[2]{\distss\left (#1,#2 \right )}
\newcommand{\distssmall}[2]{\distss (#1,#2 )}

\newcommand{\T}{^\top}

\newcommand{\E}{\mathbb{E}}

\newcommand{\lip}{L}

\newcommand\oprocendsymbol{\hbox{$\square$}}
\newcommand\oprocend{\relax\ifmmode\else\unskip\hfill\fi\oprocendsymbol}

\def\er/{Erd\H{o}s-R\'enyi}

\def\algoext/{ADMM-Tracking Gradient}
\def\ralgoext/{Robust ADMM-Tracking Gradient}

\def\algo/{ATG}
\def\ralgo/{RATG}

\def\algogameext/{Primal TRacking-based Aggregative Distributed Equilibrium Seeking}
\def\algogame/{Primal TRADES}

\newcommand{\m}{m}

\newcommand{\step}{\gamma}

\newcommand{\rand}{r}
\newcommand{\rt}{\rand\ud{\iter}}

\newcommand{\nr}{p}

\newcommand{\as}{\text{ a. s.}}

\newcommand{\cost}{\ell}

\newcommand{\dyn}{f}

\newcommand{\fast}{\cF}
\newcommand{\slow}{\cS}

\newcommand{\red}{\cR}

\def\PL/{Polyak-\L{}ojasiewicz}

\newcommand{\xddots}{%
	\raise 4pt \hbox {.}
	\mkern 6mu
	\raise 1pt \hbox {.}
	\mkern 6mu
	\raise -2pt \hbox {.}
}

\graphicspath{{figs/}}

\def\er/{Erd\H{o}s-R\'enyi}

\newcommand{\cf}{c_F}

\newcommand{\cs}{c_S}
\newcommand{\ocs}{(1-\cs)}

\newcommand{\lips}{\lip_{\slow}}
 
\newcommand{\lipr}{\lip_{\red}} 
\newcommand{\lipeq}{\lip_{\text{\tiny FIX}}}

\newcommand{\Z}{Z}

\newcommand{\Xeq}{X_{\text{\tiny fix}}} 
\newcommand{\Zeq}{Z_{\text{\tiny FIX}}} 

\newcommand{\bzx}{\bz_{x}}

\newcommand{\bzxz}{\bz_{(\x,\z)}}

\newcommand{\bxx}{\bx_{\x}}

\newcommand{\reg}{\mu}

\newcommand{\op}{G}

\def\er/{Erd\H{o}s-R\'enyi}

\newcommand{\zss}{\z_{\text{ss}}}

\newcommand{\ustar}{u_{\star}}

\newcommand{\costzstar}{\cost^{\zss}_\star}

\newcommand{\xp}{\x\ud{+}}
\newcommand{\zp}{\z\ud{+}}

\newtheorem{theorem}{Theorem}
\newtheorem{proposition}{Proposition}
\newtheorem{corollary}{Corollary}

\newtheorem{lemma}{Lemma}

\newtheorem{assumption}{Assumption}
\newtheorem{remark}{Remark}

\title{Timescale Separation\\ Through the Lens of Operator Theory}

\author{Guido Carnevale,
Nicola Bastianello,  
Luca Schenato, 
Giuseppe Notarstefano, 
Ruggero Carli, 
\thanks{Work supported in part by Fondi PNRR - Bando PE - Progetto
  PE11 - 3A-ITALY, ``Made in Italy Circolare e Sostenibile'' - Codice
  PE0000004, CUP: J33C22002950001.}
\thanks{Guido Carnevale and Giuseppe Notarstefano are with the Department of Electrical,  Electronic and Information Engineering,  Alma Mater Studiorum - Universit\`a di Bologna,  Bologna, Italy. Emails: {\tt\footnotesize{name.surname@unibo.it}}.
}
	\thanks{Nicola Bastianello is with the School of Electrical Engineering and Computer Science, and Digital Futures, KTH Royal Institute of Technology, Stockholm, Sweden. Email: {\tt\footnotesize nicolba@kth.se}.}%
\thanks{Ruggero Carli and Luca Schenato are with the Department of Information Engineering of the University of Padova, Via G. Gradenigo 6/B, 35131 Padova, Italy. Email: {\tt\footnotesize{\{carlirug,schenato\}@dei.unipd.it}}.
}
}

\begin{document}

\maketitle
\thispagestyle{empty}

\begin{abstract}
  Timescale separation is a powerful tool for analyzing interconnected dynamical systems. 
  Meanwhile, operator theory provides a general framework for studying the convergence of iterative methods formulated as fixed-point iterations, including algorithms arising in optimization, learning, and control.
  In this paper, we bridge these two areas by establishing timescale separation results for fixed-point iterations induced by both deterministic and stochastic operators.
  As customary in timescale separation, our results involve auxiliary systems that arise from the original interconnection in the limit as the timescale parameter tends to zero and separately capture the dynamics induced by the slow and fast operators.
  The proposed operator-theoretic framework yields explicit and readily checkable bounds on this tunable parameter, expressed in terms of standard operator constants. %
  To illustrate the applicability of our results, we employ them to prove the convergence properties of a feedback optimization scheme in both deterministic and stochastic settings.
\end{abstract}

\section{Introduction}
\label{sec:intro}

Timescale separation (or singular perturbation) theory is a fundamental principle in the analysis and design of interconnected dynamical systems in the form 
\begin{subequations}\label{eq:SP_system}
  \begin{align}
    \xtp &= \xt + \pr\slow(\xt,\zt) 
    \label{eq:SP_system_slow}
    \\
    \ztp &= \fast(\xt,\zt),
    \label{eq:SP_system_fast}
  \end{align}
\end{subequations}
where $\xt \in \R^{\n}$ and $\zt \in \R^{\m}$ are the system states, $\pr > 0$ is a tuning parameter, with $\slow: \R^{\n} \times \R^{\m} \to \R^{\n}$ and $\fast: \R^{\n} \times \R^{\m} \to \R^{\m}$. %
With an eye to typical nomenclature in the field of singular perturbations, we refer to~\eqref{eq:SP_system_slow} and $\xt$ as the \emph{slow} system and state, respectively, and, analogously, to~\eqref{eq:SP_system_fast} and $\zt$ as the \emph{fast} system and state, respectively.
The reason behind this nomenclature is that, since $\pr$ is tunable, the changes in $\xt$ across iterations $\iter$ can be made arbitrarily small, whereas $\zt$ may still vary even in the limiting case $\pr \to 0$.
Fig.~\ref{fig:bd_interconnection} graphically depicts system~\eqref{eq:SP_system} through a block diagram.
\begin{figure}[H]
  \centering
  \includegraphics{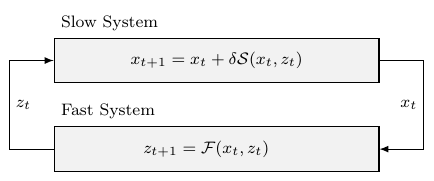}
  \caption{Graphical representation of system~\eqref{eq:SP_system}.}
  \label{fig:bd_interconnection}
\end{figure}
This peculiar class of systems can often be conveniently analyzed through two auxiliary descriptions: a boundary layer subsystem, obtained by freezing the slow variables, and a reduced subsystem, obtained by replacing in the slow dynamics the fast variables with their limiting response.
This idea lies at the core of singular perturbation theory and has played a central role in control, see the seminal works~\cite{kokotovic1968singular,sannuti1969near,kokotovic1976singular} and the comprehensive surveys~\cite{kokotovic1999singular,abdelgalil2023multi}.
Unlike the continuous-time setting, where timescale separation (or singular perturbation) theory is well established, comparatively few results are available for discrete-time systems, despite the fact that many relevant applications are naturally modeled or implemented in discrete time. 
Most existing contributions rely on Lyapunov-based arguments, or more generally on stability-theoretic tools~\cite{bai1988averaging,bouyekhf1997analysis,teel2003unified,kurina2017discrete}. 
Although these approaches provide valuable qualitative guarantees, they often yield purely existential conditions or highly conservative bounds whose complex expressions make them difficult to verify or exploit in practice.
The situation is even more limited in stochastic settings, for which only a handful of discrete-time timescale-separation results are currently available, with~\cite{carnevale2024timescale} being one of the few examples. 
The results in~\cite{carnevale2024timescale}, however, are again primarily stability-based and therefore inherit the same limitations in terms of conservatism, interpretability, and practical applicability.

In parallel, operator-theoretic methods have become a standard language for studying discrete-time iterative algorithms in optimization, learning, signal processing, and control~\cite{lohmiller1998contraction,aminzare2014contraction,yi2019operator,belgioioso2021semi,franci2021stochastic,belgioioso2022distributed,giaccagli2024synchronization,davydov2024perspectives,bianchi2024end,fabiani2026finite}. 
Indeed, the update law of many algorithms of interest can be naturally represented as a fixed-point iteration associated with an operator satisfying structural properties such as contractivity, nonexpansiveness, averagedness, paracontractivity, or metric subregularity~\cite{ryu2016primer,bauschke2017convex,nguyen2018contraction,themelis2019supermann,BASTIANELLO2026630,centorrino2024weakly,davydov2024non,davydov2025time}. 
This point of view is particularly useful because it separates the convergence mechanism from the specific algebraic form of the algorithm, thereby enabling unified analyses across apparently different schemes.

While operator-theoretic methods and continuous-time timescale-separation theory are both well developed, their combination in the analysis of interconnected discrete-time fixed-point iterations remains largely unexplored. 
This gap calls for a general operator-theoretic theory of timescale separation in discrete time, capable of yielding explicit, interpretable, and practically verifiable conditions for deterministic and stochastic operator interconnections. 
Such a theory would be directly relevant to a broad range of algorithmic architectures arising in optimization, learning, control, and networked systems, as illustrated by the following representative examples.
In feedback optimization, an optimization variable is updated using measurements generated by a physical plant that only asymptotically tracks its steady-state response~\cite{colombino2019online,hauswirth2020timescale,hauswirth2021optimization,carnevale2023nonconvex}. 
In suboptimal model predictive control, the optimization problem is solved iteratively while the plant evolves in closed loop with the current suboptimal solution~\cite{zanelli2021lyapunov,chen2025sampled,di2026suboptimal}.
In distributed optimization and equilibrium-seeking algorithms, local decision variables are updated through optimization steps, while auxiliary consensus or tracking variables evolve concurrently through communication protocols~\cite{carnevale2023admm,carnevale2024tracking,carnevale2024unifying}. 
In all these cases, the overall scheme can naturally be viewed as the interconnection of a slow operator and a fast operator, where the fixed-point set of the latter depends on the current value of the slow state.
Moreover, in more complex settings, coordination and optimization updates may be affected by asynchronous updates, packet losses, or quantized communications~\cite{bastianello2020asynchronous,bastianello2022admm,carnevale2025modular}. 
All these phenomena can be conveniently modeled through stochastic operators, which motivates the need for a stochastic extension of timescale separation results.

This paper develops a timescale-separation theory for interconnected operator dynamics. 
We consider discrete-time systems in which the slow state is updated through an operator scaled by a tunable parameter, while the fast state evolves according to an operator whose fixed-point manifold is parametrized by the slow state. 
The reduced dynamics is obtained by constraining the fast state to this manifold. 
The main question addressed in the paper is: under which conditions on the fast operator and the reduced slow operator does the original interconnection converge for sufficiently small values of the timescale parameter?
Our operator-theoretic formulation provides a constructive answer to this question by yielding explicit bounds on the timescale parameter that are readily checkable from the operator constants appearing in the assumptions.

The main theoretical results of the paper are as follows.
First, we establish a deterministic timescale-separation result for the case in which the reduced operator is contractive. Under a paracontraction property of the fast operator with respect to its fixed-point manifold, Lipschitz regularity of the interconnection, and contractivity of the reduced dynamics, we prove linear convergence of the full interconnected system. 
The proof yields an explicit and readily checkable upper bound on the admissible timescale parameter, expressed in terms of the contraction rates and interconnection gains appearing in the assumptions.
Second, we extend the deterministic analysis beyond contractive reduced dynamics. Specifically, we consider the case in which the reduced operator is averaged and metrically subregular. This setting covers a broader class of fixed-point iterations commonly arising in operator-splitting and optimization algorithms. We show that averagedness together with metric subregularity implies a linear decrease of the distance to the fixed-point set of the reduced operator, and we use this property to establish linear convergence of the original interconnected dynamics.
Third, we develop a stochastic extension of the framework. We consider random slow and fast operators driven by an independent stochastic process and derive sufficient conditions ensuring almost sure convergence of the interconnected system. The assumptions are formulated in terms of expected contraction-type properties of the fast operator and of the reduced stochastic operator. This result allows the framework to cover algorithms affected by random updates, asynchronous activations, or packet losses.
In contrast with many qualitative singular-perturbation formulations, the proposed operator-theoretic analysis yields directly computable sufficient bounds in terms of standard operator constants.
Their applicability is illustrated through feedback optimization schemes in both deterministic and stochastic settings.
Overall, the paper provides a compositional operator-theoretic principle for timescale separation: convergence of the full interconnection is obtained by combining convergence of the fast dynamics toward a parametrized fixed-point set, convergence of the reduced iteration, and explicit bounds on the coupling between the two dynamics.
In this sense, the proposed framework complements classical singular-perturbation and two-timescale analyses by addressing fixed-point iterations and operator-based algorithms directly.

The work most closely related to ours is~\cite{cothren2023online}, where contraction theory and singular-perturbation methods are combined to study continuous-time interconnected systems.
However, the results in~\cite{cothren2023online} do not directly extend to discrete-time systems.
Moreover, our framework goes beyond contractive mappings by covering averaged operators and stochastic operator dynamics.

A preliminary, shorter version of this work appeared in~\cite{carnevale2026contraction}.
The present paper substantially extends the preliminary results in several directions.
First, it provides a more general deterministic analysis that also covers scenarios in which the reduced operators are not contractive.
Second, it develops a stochastic extension of the proposed framework.
These theoretical developments are also reflected in the numerical applications, which now include both a noncontractive scenario and a stochastic setting.
Finally, the present paper provides complete proofs of all the main results, which were omitted from~\cite{carnevale2026contraction}.

The paper is organized as follows. Section \ref{sec:setup} introduces the operator interconnection framework and the standing assumptions. Sections \ref{sec:Contractive} and \ref{sec:AveReduced} present the deterministic convergence results for contractive and averaged reduced operators, respectively. Section \ref{sec:stochastic_case} develops the stochastic extension.
Finally, in Section~\ref{sec:feedback_optimization}, we apply our results to the framework of feedback optimization. %
Conclusions are drawn in Section~\ref{sec:conclusions}.
The proofs of the main results are provided in the appendices.

\paragraph*{Notation}

A square matrix $M \in \R^{n\times n}$ is said to be Schur if all its eigenvalues lie in the open unit circle.
The identity matrix in $\R^{m\times m}$ is $I_m$.
For vectors of any compatible dimension, we denote by $\norm{\cdot}$ and $\dist{\cdot}{\cdot}$ the Euclidean norm and distance, respectively.

\section{Technical Assumptions and Problem Formulation}
\label{sec:setup}

In this paper, we study the convergence properties of the interconnected system~\eqref{eq:SP_system} through an approach that bridges the perspectives of timescale separation and operator theory.
Combining the terminology of these two areas, we refer to $\slow$ and $\fast$ as the \emph{slow} and \emph{fast} operators, respectively.
We now state the standing assumptions used in the deterministic analysis. 
The first one simply ensures that the dynamics remain within prescribed forward-invariant sets $\X$ and $\Z$.
\begin{assumption}[Forward Invariance]\label{ass:forward_invariance}
  There exist closed sets $\X \subseteq \R^{\n}$ and $\Z \subseteq \R^{\m}$ such that
  \begin{align}
    \x + \pr\slow(\x,\z) \in \X, \quad \fast(\x,\z) \in \Z,
  \end{align}
  for all $\x \in \X$, $\z \in \Z$, and $\pr \in [0,1]$.\oprocend
\end{assumption}
We highlight that no boundedness of $\X$ and $\Z$ is required. 
In particular, Assumption~\ref{ass:forward_invariance} is trivially satisfied with $\X=\R^{\n}$ and $\Z= \R^{\m}$ whenever it is not needed to restrict the operators to specific domains.

Before proceeding, let us introduce two norms $\normf{\cdot}$ and $\norms{\cdot}$ induced by inner products on $\R^m$ and $\R^n$, respectively, and define the corresponding distances $\distfs: \R^m \times \R^m \to \R$ and $\distss: \R^n \times \R^n \to \R$ as 
\begin{align*}
  \distf{\z}{\z^\prime} &:= \normf{\z - \z^\prime}
  \\
  \dists{\x}{\x^\prime} &:= \norms{\x - \x^\prime}.
\end{align*}
Moreover, given $\Z^\prime \subseteq \Z$, $X^\prime \subseteq \X$, $\z \in \Z$, and $\x \in \X$, we define the distance of $\z$ from $Z^\prime$ and of $\x$ from $X^\prime$ as 
\begin{align*}
  \distf{\z}{Z^\prime} &:= \inf_{y \in Z^\prime}\distf{\z}{y}
  \\
  \dists{\x}{X^\prime} &:= \inf_{y \in X^\prime}\dists{\x}{y}.
\end{align*}

The key structural feature of the interconnection is that, for each fixed value of the slow state $\x$, the fast operator $\fast$ admits a fixed-point set $\Zeq(\x)$. %
This set contains the limiting response of the fast dynamics when the slow variable is held constant.
\begin{assumption}[Fixed-point manifold of the fast operator]\label{ass:set_fix_parametrized}
  There exists a set-valued map $\Zeq: \X \rightrightarrows \Z$ such that
  \begin{align}
    \bar{\z} = \fast(\x,\bar{\z}),
  \end{align} 
  for all $\x \in \X$ and $\bar{\z} \in \Zeq(\x)$.
  Moreover, $\Zeq(\x)$ is nonempty and closed for all $\x \in \X$. 
  Further, there exists $\lipeq > 0$ such that 
  \begin{align}\label{eq:lipschitz_eq}
   \distf{\z}{\Zeq(x^\prime)} \leq \lipeq\dists{\x}{\x^\prime},
  \end{align}
  for all $\x, \x^\prime \in \X$ and $\z \in \Zeq(\x)$.\oprocend
\end{assumption}
Assumption \ref{ass:set_fix_parametrized} states that the fixed-point set of the fast operator is well defined for every frozen slow state and varies Lipschitz continuously with that state. The constant $\lipeq$ quantifies how much the target set of the fast dynamics moves when the slow state changes. 

The following assumption ensures that the fast operator $\fast$ is \emph{paracontractive} (see, e.g.,~\cite[Def.~4]{deplano2025optimization}) within $\Z$.
We note that paracontractivity is weaker than contractivity, as it only requires the distance from any fixed point to strictly decrease, rather than the distance between arbitrary pairs of points.
If the set $\Zeq(\x) = \{\zeq(\x)\}$ is a singleton for every $\x \in \X$, then paracontractivity reduces to contractivity with respect to the unique fixed point $\zeq(\x)$.
\begin{assumption}[Uniform contraction toward the fast fixed-point set] 
\label{ass:fast_paracontractive} 
  There exists $\cf \in (0,1)$ such that 
  \begin{align}\label{eq:paracontraction}
    \distf{\fast(\x,\z)}{\Zeq(\x)} \leq \cf\distf{\z}{\Zeq(\x)}, 
  \end{align}
  for all $(\x,\z) \in \X\times \Z$.\oprocend
\end{assumption}
This assumption represents the counterpart, in our framework, of the standard stability assumption imposed in the singular perturbation literature on the so-called \emph{boundary-layer system}.
Indeed, inequality~\eqref{eq:paracontraction} is formulated for an arbitrarily fixed value of the slow state $\x$.
Accordingly, Assumption~\ref{ass:fast_paracontractive} can be interpreted as characterizing the properties of the fast dynamics $\fast$ when $\pr=0$ in~\eqref{eq:SP_system}.
This trivially yields $\xtp=\xt$ and leads to the auxiliary system depicted in Fig.~\ref{fig:bd_bl}.
\begin{figure}[H]
  \centering
  \includegraphics{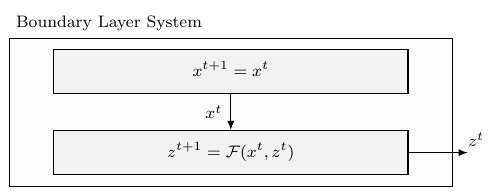}
  \caption{Graphical representation of the boundary layer system (characterized in Assumption~\ref{ass:fast_paracontractive}) associated to the interconnected system~\eqref{eq:SP_system}.}
  \label{fig:bd_bl}
\end{figure}
We now focus on the ideal limiting case in which subsystem~\eqref{eq:SP_system_fast} is infinitely faster than subsystem~\eqref{eq:SP_system_slow} and, thus, its state $\zt$ is always in the fixed-point set $\Zeq(\xt)$ associated to the current slow state $\xt$.
The resulting dynamics is typically referred to as the \emph{reduced system} in the singular perturbation literature.
In particular, we impose the following assumption to ensure the existence of a \emph{reduced} operator $\red(\x)$ that fully describes $\slow(\x, \z)$ when $\z \in \Zeq(\x)$, namely, when the fast state lies in the equilibrium manifold $\Zeq(\x)$ associated to the slow state $\x$.
\begin{assumption}[Reduced operator]\label{ass:reduced_operator}
  There exists $\red: \X \to \R^{\n}$ such that
  \begin{align}
    \red(\x) = \slow(\x,\z),
  \end{align}
  for all $(\x,\z) \in \{(\x,\z) \in \X \times \Z \mid \x \in \X, \z \in \Zeq(x)\}$.\oprocend
\end{assumption}
Based on the above definition, we can introduce the so-called \emph{reduced} system associated to~\eqref{eq:SP_system}, that is
\begin{align}\label{eq:reduced_system}
  \xtp  = \xt + \red(\xt).
\end{align}
We note that the tunable parameter $\pr$ does not appear in the reduced system~\eqref{eq:reduced_system}, which is equivalent to setting $\pr=1$.
The underlying intuition is that, in the ideal limiting regime in which $\zt \in \Zeq(\xt)$ for every $\iter \in \N$, there is no need to slow down the evolution of $\xt$ to ensure convergence.
We graphically depict the reduced system~\eqref{eq:reduced_system} in Fig.~\ref{fig:bd_reduced}.
\begin{figure}[H]
  \centering
  \includegraphics{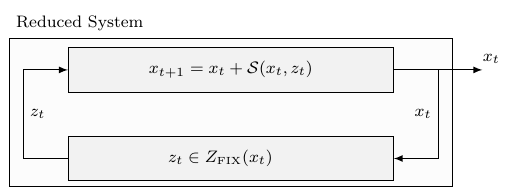}
  \caption{Graphical representation of the reduced system~\eqref{eq:reduced_system} associated to the interconnected system~\eqref{eq:SP_system}.}
  \label{fig:bd_reduced}
\end{figure}
Finally, we impose the following Lipschitz continuity conditions.
\begin{assumption}[Lipschitz continuity]\label{ass:lip}
  There exist $\lips, \lipr > 0$ such that
  \begin{align*}
    \norms{\slow(\x,\z) - \slow(\x,\z')} &\leq \lips\normf{\z - \z'}
    \\
    \norms{\red(\x) - \red(\x')} &\leq \lipr \norms{\x - \x'},
  \end{align*}
  for all $(\x,\z), (\x',\z') \in \X \times \Z$.
  \oprocend
\end{assumption}
The constants $\lips$ and $\lipr$ quantify the coupling between the fast transient and the slow update, and the sensitivity of the reduced slow operator. 

The goal of this paper is to analyze the convergence properties of the interconnected dynamical system in \eqref{eq:SP_system} under different assumptions on the reduced operator $\x+\red(\x)$. 
In particular, in Section \ref{sec:Contractive}, we assume $\x+\red(\x)$ to be contractive, whereas in Section \ref{sec:AveReduced}, we assume $\x+\red(\x)$ to be averaged and metrically subregular.
In both cases, we establish conditions on $\delta$ guaranteeing that the state trajectory converges to a fixed-point set. 
Finally, in Section \ref{sec:stochastic_case}, we extend our results to stochastic scenarios.

\section{Deterministic Framework: Contractive Case} \label{sec:Contractive}

We first consider the case in which the reduced dynamics~\eqref{eq:reduced_system} is contractive. 
This setting yields the cleanest form of timescale separation result and illustrates the main proof mechanism. 
\begin{assumption}[Reduced system contractive]\label{ass:slow_contractive}
  There exists $\cs \in (0,1)$ such that
  \begin{align}\label{eq:slow_paracontractive}
    \dists{\x + \red(\x)}{\x^\prime+\red(\x^\prime)} \leq \ocs\dists{\x}{\x^\prime},
  \end{align}
  for all $\x, \x^\prime \in \X$. 
  \oprocend
\end{assumption}
We remark that Assumption~\ref{ass:slow_contractive} implies that $x + \red(x)$ has a unique fixed point and we denote it by $\xstar \in \X$.

The next theorem provides a bound on $\pr$ that guarantees linear convergence of~\eqref{eq:SP_system} toward the set $\{(\x,\z)\in \X \times \Z \mid \x = \xstar, \z \in \Zeq(\xstar)\}$.
\begin{theorem}[Timescale separation I]\label{th:contraction}
  Consider~\eqref{eq:SP_system} and let Assumptions~\ref{ass:forward_invariance},~\ref{ass:set_fix_parametrized},~\ref{ass:fast_paracontractive},~\ref{ass:reduced_operator},~\ref{ass:lip} and~\ref{ass:slow_contractive} hold.
  Define
  \begin{align}\label{eq:bar_pr}
    \bar{\pr} := \frac{\cs(1 -\cf)}{\lipeq\lips(\cs + \lipr)}.
  \end{align}
  Then, for every $\pr \in (0,\min\{\bar{\pr},1\})$, there exist $C > 0$ and $\rho \in (0,1)$ such that the trajectories of~\eqref{eq:SP_system} satisfy
  \begin{align}
  \left\|
       \begin{bmatrix}
          \dists{\xt}{\xstar}
          \\
          \distf{\zt}{\Zeq(\xt)}
        \end{bmatrix}
        \right\|
        \leq 
        C \rho^\iter
          \left\|\begin{bmatrix}
          \dists{\x_0}{\xstar}
          \\
          \distf{\z_0}{\Zeq(\x_0)}
        \end{bmatrix}\right\|,\label{eq:linear_rate}
  \end{align}
  for all initial conditions $(\x\ud0,\z\ud0) \in \X \times \Z$ and $\iter \in \N$. In particular, $\xt \to \xstar$ and $\distf{\zt} {\Zeq(\xt)} \to 0$ linearly. \oprocend
\end{theorem}
The proof of Theorem~\ref{th:contraction} is provided in Appendix~\ref{sec:proof_th_contraction}.

\begin{remark}
[Interpretation of the bound] The admissible upper bound~\eqref{eq:bar_pr} on $\pr$ decreases when the fast fixed-point manifold is more sensitive to the slow state (i.e., the larger $\lipeq$ is), when the slow operator is more sensitive to fast-state errors (i.e., the larger $\lips$ is), or when the reduced operator is more sensitive to changes in the slow state (i.e., the larger $\lipr$ is). Conversely, larger contraction margins of the fast and reduced dynamics (i.e., the larger $\cs$ and $1-\cf$ are) allow larger values of $\pr$.
\end{remark}

\section{Deterministic Framework: Averaged and Metrically Subregular Case}
\label{sec:AveReduced}

The contractivity requirement in Assumption \ref{ass:slow_contractive} is often too strong for algorithms arising from operator splitting or projected fixed-point iterations. We therefore consider the case in which the reduced operator is averaged and metrically subregular. This combination is sufficient to recover a linear decrease of the distance to the fixed-point set, which plays the same role as contractivity in the interconnection analysis. 
To properly state this assumption, let $$\Xeq := \{\x \in \X \mid \red(\x) = 0\}$$ 
denote the set of fixed points of the reduced operator.
\begin{assumption}[Reduced system averaged and metrically subregular]\label{ass:averaged_and_metric_subregular}
  The set $\Xeq$ is nonempty and closed. 
  Further, the operator $\x + \red(\x)$ is $\alpha$-averaged and $\reg$-metric subregular within $\X$ and with respect to $\norms{\cdot}$, for some $\alpha \in (0,1)$ and $\reg \ge 1$, namely 
  \begin{subequations}
    \begin{align}
      \dists{\x + \frac{1}{\alpha}\red(\x)}{\x^\prime + \frac{1}{\alpha}\red(\x^\prime)} &\leq\norms{\x - \x^\prime}
      \label{eq:non_expansive}
      \\
      \dists{\x}{\Xeq} &\leq \reg\norms{\red(\x)},
      \label{eq:metric_subregularity}
    \end{align}
  \end{subequations}
  for all $\x, \x^\prime \in \X$. \oprocend
\end{assumption}

The next lemma makes explicit the contraction-like property implied by Assumption \ref{ass:averaged_and_metric_subregular}.
\begin{lemma}[Reduced system linear convergence]\label{lemma:linear_convergence_reduced_operator}
  Let Assumption~\ref{ass:averaged_and_metric_subregular} hold.
  Then, with $\cs := 1 - \sqrt{1 - \tfrac{1 - \alpha}{\alpha\reg^2}}$, it holds
  \begin{align}\label{eq:linear_convergence_without_contraction}
    \dists{\x + \red(\x)}{\bxx} \leq (1 - \cs)\dists{\x}{\bxx},
  \end{align}
  for all $\x \in \X$ and $\bxx \in \arg\inf_{y \in \Xeq}\dists{\x}{y}$. \oprocend
\end{lemma}
The proof of Lemma~\ref{lemma:linear_convergence_reduced_operator} is provided in Appendix~\ref{sec:proof_lemma_linear_convergence_reduced_operator}.
With Lemma~\ref{lemma:linear_convergence_reduced_operator} at hand, we are ready to establish the convergence properties of the interconnected system~\eqref{eq:SP_system} for averaged and metrically subregular reduced dynamics.
\begin{theorem}[Timescale separation II]\label{th:linear_convergence_without_contraction}
  Consider~\eqref{eq:SP_system} and let Assumptions~\ref{ass:forward_invariance},~\ref{ass:set_fix_parametrized},~\ref{ass:fast_paracontractive},~\ref{ass:reduced_operator},~\ref{ass:lip} and~\ref{ass:averaged_and_metric_subregular} hold. Let $\cs$ be as in Lemma \ref{lemma:linear_convergence_reduced_operator}.
  Define
  \begin{align}\label{eq:bar_pr_ave}
    \bar{\pr} := \frac{\cs(1 -\cf)}{\lipeq\lips(\cs + \lipr)}.
  \end{align}
  Then, for all $\pr \in (0,\min\{\bar{\pr},1\})$, there exist $C > 0$ and $\rho \in (0,1)$ such that the trajectories of~\eqref{eq:SP_system} satisfy
  \begin{align*}
  \left\|
       \begin{bmatrix}
          \dists{\xt}{\Xeq}
          \\
          \distf{\zt}{\Zeq(\xt)}
        \end{bmatrix}\right\|
        \leq 
        C \rho^\iter \left\|
          \begin{bmatrix}
          \dists{\x_0}{\Xeq}
          \\
          \distf{\z_0}{\Zeq(\x_0)}
        \end{bmatrix}\right\|,
  \end{align*}
  for all $(\x\ud0,\z\ud0) \in \X \times \Z$ and $\iter \in \N$. \oprocend
\end{theorem}
The proof of Theorem~\ref{th:linear_convergence_without_contraction} is provided in Appendix~\ref{sec:proof_th_linear_convergence_without_contraction}.

\section{Stochastic Scenario}
\label{sec:stochastic_case}

We now extend the framework to stochastic scenarios.
To this end, we consider a stochastic counterpart of system~\eqref{eq:SP_system} described by 
\begin{subequations}\label{eq:SP_system_stochastic}
  \begin{align}
    \xtp &= \xt + \pr\slow(\xt,\zt,\rt) 
    \label{eq:SP_system_stochastic_slow}
    \\
    \ztp &= \fast(\xt,\zt,\rt),
    \label{eq:SP_system_stochastic_fast}
  \end{align}
\end{subequations}
in which now $\slow: \R^{\n} \times \R^{\m} \times \R^{\nr} \to \R^{\n}$ and $\fast: \R^{\n} \times \R^{\m} \times \R^{\nr} \to \R^{\m}$, $\{\rt\}_{\iter \in \N}$ is an i.i.d. sequence of $\cD$-valued random variables, with $\cD \subseteq \R^{\nr}$, defined on a given complete probability space, %
while the other symbols have the same meaning as in~\eqref{eq:SP_system}.
\begin{remark}[Deterministic time-varying systems]
  The framework above also admits a deterministic time-varying
  counterpart. 
  Specifically, the i.i.d. random sequence
  $\{\rt\}_{\iter\in\N}$ can be replaced by an arbitrary prescribed
  $\cD$-valued sequence.
  In this case, the pointwise assumptions remain unchanged, whereas
  each expectation-based assumption is replaced by the corresponding
  pointwise inequality, required to hold uniformly along the prescribed
  sequence.
  Under these deterministic counterparts of the assumptions, the same
  comparison argument applies pathwise and yields pointwise convergence of the
  resulting deterministic time-varying interconnection.
\end{remark}
The stochastic assumptions below are counterparts of the
deterministic ones. They are formulated so that the one-step
comparison inequalities hold conditionally on the current state.
\begin{assumption}[Stochastic forward invariance]\label{ass:forward_invariance_stoch}
  There exist closed sets $\X \subseteq \R^{\n}$ and $\Z \subseteq \R^{\m}$ such that
  \begin{align}
    \x + \pr\slow(\x,\z,\rand) \in \X, \quad \fast(\x,\z,\rand) \in \Z,
  \end{align}
  for all $(\x,\z,\rand) \in \X \times \Z \times \cD$, and $\pr \in [0,1]$.\oprocend
\end{assumption}
\begin{assumption}[Fixed-point manifold of the stochastic fast operator]\label{ass:set_fix_parametrized_stoch}
  There exists a set-valued map $\Zeq: \X \rightrightarrows \Z$ such that
  \begin{align}
    \bar{\z} = \fast(\x,\bar{\z},\rand),
  \end{align}
  for all $\rand \in \cD$, $\x \in \X$, and $\bar{\z} \in \Zeq(\x)$.
  Moreover, $\Zeq(\x)$ is nonempty and closed for all $\x \in \X$.
  Further, there exists $\lipeq > 0$ such that 
  \begin{align}\label{eq:lipschitz_eq_stoch}
   \distf{\z}{\Zeq(x^\prime)} \leq \lipeq\dists{\x}{\x^\prime},
  \end{align}
  for all $\x, \x^\prime \in \X$ and $\z \in \Zeq(\x)$.\oprocend
\end{assumption}
\begin{assumption}[Stochastic contraction toward the fast fixed-point manifold]\label{ass:fast_paracontractive_stoch} 
  There exists $\cf \in (0,1)$ such that 
  \begin{align}
    \E[\distf{\fast(\x,\z,\rand)}{\Zeq(\x)}] \leq \cf\distf{\z}{\Zeq(\x)},
  \end{align}
  for all $(\x,\z) \in \X\times \Z$.\oprocend
\end{assumption}
\begin{assumption}[Stochastic reduced operator]\label{ass:reduced_operator_stoch}
  There exists $\red: \R^{\n} \times \cD \to \R^{\n}$ such that
  \begin{align}
    \red(\x,\rand) = \slow(\x,\z,\rand),
  \end{align}
  for all $(\x,\z,\rand) \in \{(\x,\z,\rand) \in \X  \times \Z \times \cD\mid \x \in \X, \z \in \Zeq(x)\}$.\oprocend
\end{assumption}
\begin{assumption}[Stochastic Lipschitz
continuity]\label{ass:lip_stoch}
  There exist $\lips, \lipr > 0$ such that
  \begin{subequations}\label{eq:lip_stoch}
    \begin{align}
            \norms{\slow(\x,\z,\rand) - \slow(\x,\z^\prime,\rand)}&\leq \lips\normf{\z - \z^\prime} 
      \\
      \norms{\red(\x,\rand) - \red(\x^\prime,\rand)} &\leq \lipr\norms{\x - \x^\prime},
    \end{align}
  \end{subequations}
  for all $(\x,\z), (\x^\prime,\z^\prime)\in \X \times \Z$ and $\rand \in \cD$.
  \oprocend
\end{assumption}
Now let us introduce the fixed-point set $\Xeq \subseteq \X$ of the reduced operator $\x + \red(\x,\rand)$, namely, 
$$\Xeq := \{\x \in \X \mid \red(\x,\rand) = 0 \text{ for all } \rand \in \cD\}.$$
\begin{assumption}[Reduced system stochastically averaged and metrically subregular]\label{ass:averaged_and_metric_subregular_stoch}
  The set $\Xeq$ is nonempty and closed. 
  Further, the stochastic operator $\x + \red(\x,\rand)$ is $\alpha$-averaged and $\reg$-metric subregular within $\X$ and with respect to $\norms{\cdot}$, for some $\alpha \in (0,1)$ and $\reg \ge 1$, namely 
  \begin{subequations}
    \begin{align}
      \distssmall{\x + \tfrac{1}{\alpha}\red(\x,\rand)}{\x^\prime + \tfrac{1}{\alpha}\red(\x^\prime,\rand)} &\leq\norms{\x - \x^\prime}
      \label{eq:non_expansive_stoch}
      \\
      \dists{\x}{\Xeq} &\leq \reg\E[\norms{\red(\x,\rand)}],\label{eq:metric_subregularity_stoch}
    \end{align}
  \end{subequations}
  for all $\x, \x^\prime \in \X$ and $\rand \in \cD$. 
  \oprocend
\end{assumption}
We are now ready to provide the stochastic counterpart of Lemma~\ref{lemma:linear_convergence_reduced_operator}.
\begin{lemma}[Reduced system linear convergence]\label{lemma:linear_convergence_reduced_operator_stoch}
  Let Assumption~\ref{ass:averaged_and_metric_subregular_stoch} hold.
  Then, with $\cs := 1 - \sqrt{1 - \tfrac{1 - \alpha}{\alpha\reg^2}}$, it holds
  \begin{align}
    \E\left[\dists{\x + \red(\x,\rand)}{\bxx}\right] \leq (1 - \cs)\dists{\x}{\bxx},
  \end{align}
  for all $\x \in \X$ and $\bxx \in \arg\inf_{y \in \Xeq}\dists{\x}{y}$. \oprocend
\end{lemma}
The proof of Lemma~\ref{lemma:linear_convergence_reduced_operator_stoch} is provided in Appendix~\ref{sec:proof_lemma_linear_convergence_reduced_operator_stoch}.
With Lemma~\ref{lemma:linear_convergence_reduced_operator_stoch} at hand, we are ready to establish the convergence result for the stochastic interconnected system~\eqref{eq:SP_system_stochastic}.
\begin{theorem}[Stochastic timescale separation I]
  \label{th:linear_convergence_without_contraction_stoch}
Let Assumptions~\ref{ass:forward_invariance_stoch},~\ref{ass:set_fix_parametrized_stoch},~\ref{ass:fast_paracontractive_stoch},~\ref{ass:reduced_operator_stoch},~\ref{ass:lip_stoch}, and~\ref{ass:averaged_and_metric_subregular_stoch} hold.
  Let 
  \begin{align}\label{eq:bar_pr_stoch}
    \bar{\pr} := \frac{\cs(1 -\cf)}{\lipeq\lips(\cs + \lipr)}.
  \end{align}
  Then, for all $\pr \in (0,\min\{\bar{\pr},1\})$, there exist $C > 0$ and $\rho \in (0,1)$ such that the trajectories of~\eqref{eq:SP_system_stochastic} satisfy
  \begin{align}
    \norm{\begin{bmatrix}
      \E[\dists{\xt}{\Xeq}]
      \\
      \E[\distf{\zt}{\Zeq(\xt)}]
    \end{bmatrix}} 
    \leq C\rho^\iter
    \norm{\begin{bmatrix}
      \dists{\x_0}{\Xeq}
      \\
      \distf{\z_0}{\Zeq(\x_0)}
    \end{bmatrix}},
    \label{eq:statement_stochastic_case}
  \end{align}
  for all $(\x\ud0,\z\ud0) \in \X \times \Z$ and $\iter \in \N$. \oprocend
\end{theorem}
The proof of Theorem~\ref{th:linear_convergence_without_contraction_stoch} is provided in Appendix~\ref{sec:proof_th_linear_convergence_without_contraction_stoch}.
The next corollary studies the case in which the reduced system is contractive rather than satisfying Assumption~\ref{ass:averaged_and_metric_subregular_stoch}.
\begin{corollary}[Stochastic timescale separation II]
  \label{cor:linear_convergence_stoch}
  Let Assumptions~\ref{ass:forward_invariance_stoch},~\ref{ass:set_fix_parametrized_stoch},~\ref{ass:fast_paracontractive_stoch},~\ref{ass:reduced_operator_stoch}, and~\ref{ass:lip_stoch} hold.
  Moreover, assume there exist $\xstar \in \X$ and $\cs \in (0,1)$ such that $\Xeq = \{\xstar\}$ and
  \begin{align}\label{eq:reduced_contraction_stoch}
    \E[\dists{x+\red(x,\rt)}{\xstar}] \leq (1-\cs)\dists{x}{\xstar},
  \end{align}
  for all $\x \in \X$.
  Then, for all $\pr \in (0,\min\{\bar{\pr},1\})$ (cf.~\eqref{eq:bar_pr_stoch}), there exist $C > 0$ and $\rho \in (0,1)$ such that the trajectories of~\eqref{eq:SP_system_stochastic} satisfy
  \begin{align}
    \norm{\begin{bmatrix}
      \E[\dists{\xt}{\xstar}]
      \\
      \E[\distf{\zt}{\Zeq(\xt)}]
    \end{bmatrix}} \leq C \rho^\iter
    \norm{\begin{bmatrix}
      \dists{\x_0}{\xstar}
      \\
      \distf{\z_0}{\Zeq(\x_0)}
    \end{bmatrix}},
    \label{eq:statement_stochastic_case_contraction}
  \end{align}
  for all $(\x\ud0,\z\ud0) \in \X \times \Z$ and $\iter \in \N$, where $\xstar \in \Xeq$ is the unique fixed point of the reduced operator $\x + \red(\x,\rand)$.
\end{corollary}
The proof of Corollary~\ref{cor:linear_convergence_stoch} is provided in Appendix~\ref{sec:proof_linear_convergence_stoch}.
\begin{remark}[Almost sure convergence]
  \label{rem:almost_sure_convergence}
  By Markov's inequality and the Borel-Cantelli lemma, Theorem~\ref{th:linear_convergence_without_contraction_stoch} implies that
  \begin{align*}
    \lim_{\iter \to \infty}\norm{\begin{bmatrix}\dists{\xt}{\Xeq}
      \\
      \distf{\zt}{\Zeq(\xt)}\end{bmatrix}} = 0, \as.
  \end{align*}
  Analogously, Corollary~\ref{cor:linear_convergence_stoch} implies that
   \begin{align*}
    \lim_{\iter\to\infty}
    \norm{\begin{bmatrix}
      \dists{\xt}{\xstar}
      \\
      \distf{\zt}{\Zeq(\xt)}
    \end{bmatrix}} = 0, \as.
  \end{align*}
\end{remark}

\section{Application to Feedback Optimization}
\label{sec:feedback_optimization}

In this section, we use our results to tune controllers in the feedback-optimization setting, see~\cite{colombino2019online,hauswirth2020timescale,carnevale2023nonconvex,hauswirth2021optimization,carnevale2025almost} and references therein.
To show the applicability of all our results, we first consider the deterministic case and then turn to the stochastic case.

\subsection{Deterministic Feedback Optimization}

In the deterministic setup, we consider a nonlinear plant described by  
\begin{align}\label{eq:plant}
  \ztp = \dyn(\zt,\ut),
\end{align}
where $\zt \in \R^{\m}$ is the state of the plant, $\ut \in \R^{\n}$ is the control input, and $\dyn: \R^{\m} \times \R^{\n} \to \R^{\m}$ describes the plant dynamics.
In particular, we consider plants with a low-level controller that endows the closed-loop system with contraction properties.
Specifically, we assume there exists a \emph{steady-state} function $\zss: \R^{\n} \to \R^{\m}$ such that
\begin{align*} 
  \zss(u) = \dyn(\zss(u),u),
\end{align*} 
for all $u \in \R^{\n}$.
In feedback optimization, the goal is to steer system~\eqref{eq:plant} to a configuration corresponding to a solution of an underlying optimization problem of the form
\begin{subequations}\label{eq:problem}
    \begin{align}
      \min_{(z,u) \in \R^{\m} \times \R^{\n}} \cost(z)
      \\
      \text{ s.t. } z = \zss(u),
    \end{align}
\end{subequations}
where $\cost: \R^{\m} \to \R$ is a cost function to be minimized.
Typically, in feedback optimization, the optimization goal formalized in~\eqref{eq:problem} must be achieved without knowledge of the steady-state function $\zss$ and while simultaneously controlling system~\eqref{eq:plant}.
Ideally, one would update the control input $\ut$ by taking a descent step on the function obtained by composing the cost function $\cost$ of the constrained optimization problem~\eqref{eq:problem} with the steady-state map $\zss$.
Namely, the resulting function is $\cost(\zss(\cdot))$ and is referred to as the \emph{reduced cost function} in the feedback optimization literature.
The corresponding descent direction would thus be given by $\nabla\zss(\ut)\frac{\partial\cost(z)}{\partial z}\big|_{z = \zss(\ut)}\T$.
However, unlike the sensitivity function $\nabla\zss(\ut)$, the knowledge of $\zss(\ut)$ is not available in feedback optimization.
Therefore, the exact descent direction $\nabla\zss(\ut)\frac{\partial\cost(z)}{\partial z}\big|_{z = \zss(\ut)}\T$ cannot be used.
To overcome this issue, we replace the unavailable steady-state information $\zss(\ut)$ with the current plant state $\zt$, thus obtaining the interconnected system 
\begin{subequations}\label{eq:closed_loop}
  \begin{align}
    \utp &= \ut - \pr\step\nabla\zss(\ut)\nabla\cost(\zt)
    \label{eq:closed_loop_control_input}
    \\
    \ztp &= \dyn(\zt,\ut),\label{eq:closed_loop_plant}
  \end{align}
\end{subequations}
where $\step > 0$ is a step size parameter, while $\pr > 0$ is a parameter controlling the update rate of $\ut$.
We refer the reader to~\cite{colombino2019online,hauswirth2020timescale,carnevale2023nonconvex,hauswirth2021optimization} for more details on the feedback optimization setup and on the design of~\eqref{eq:closed_loop}.
In the next proposition, we establish the convergence properties of system~\eqref{eq:closed_loop} by leveraging Theorems~\ref{th:contraction} and~\ref{th:linear_convergence_without_contraction} in the strongly convex and convex cases, respectively.
\begin{proposition}\label{prop:feedback_opt}
  Consider system~\eqref{eq:closed_loop}.
  Assume that $\cost(\zss(\cdot))$ is differentiable, $\str$-strongly convex, and has an $\lip$-Lipschitz continuous gradient for some $\str, \lip > 0$.
  Moreover, assume that $\zss$ is differentiable and that, for all $u \in \R^{\n}$, the map $z \mapsto \dyn(z,u)$ is $\cf$-contractive with respect to $\dist{\cdot}{\zss(u)}$ for some $\cf \in (0,1)$.
  Further, let $\nabla\cost$ and $\zss$ be $\lip_\cost$- and $\lipeq$-Lipschitz continuous for some $\lip_\cost, \lipeq > 0$, respectively.
  Let 
  \begin{align}
    \bar{\pr} := \frac{(1 - \sqrt{1 - \step 2\str\lip/(\str + \lip)})(1 -\cf)}{\step\lipeq^2\lip_\cost((1 - \sqrt{1 - \step 2\str\lip/(\str + \lip)}) + \step\lip)}.\label{eq:bpr_sc}
  \end{align}
  Then, for all $\step \in (0,2/(\str + \lip)]$ and $\pr \in (0,\min\{\bar{\pr},1\})$, there exist $C > 0$ and $\rho \in (0,1)$ such that the trajectories of~\eqref{eq:closed_loop} satisfy
  \begin{align*}
    \norm{
      \begin{bmatrix}
        \dist{\ut}{\ustar}
        \\
        \dist{\zt}{\zss(\ut)}
      \end{bmatrix}} 
      \leq C \rho^\iter
      \norm{
        \begin{bmatrix}
          \dist{u\ud0}{\ustar}
          \\
          \dist{\z\ud0}{\zss(u\ud0)}
        \end{bmatrix}},
  \end{align*}
  for all $(u\ud0,\z\ud0) \in \R^{\n} \times \R^{\m}$ and $\iter \in \N$, where $\ustar \in \R^{\n}$ is the unique minimizer of $\cost(\zss(\cdot))$.
  Next, suppose that all the preceding assumptions hold, except that $\cost(\zss(\cdot))$ is only convex and satisfies the quadratic growth condition
  \begin{align*}
    \cost(\zss(u)) - \costzstar \geq \frac{\str}{2} \dist{u}{U_\star}^2,
  \end{align*}
  for some $\str > 0$, where $U_\star$ denotes the set of minimizers of $\cost(\zss(\cdot))$, and $\costzstar$ is the minimum value of $\cost(\zss(\cdot))$.
  Let 
  \begin{align}\label{eq:bpr_c}
    \bar{\pr} := \frac{\left(1 - \sqrt{1 - \frac{\step\str^2(2 - \step\lip)}{4\lip}}\right)(1 -\cf)}{\step\lipeq^2\lip_\cost\left(\left(1 - \sqrt{1 - \frac{\step\str^2(2 - \step\lip)}{4\lip}}\right) + \step\lip\right)}.
  \end{align}
  Then, for all $\step \in (0,2/\lip)$ and $\pr \in (0,\min\{\bar{\pr},1\})$, there exist $C > 0$ and $\rho \in (0,1)$ such that the trajectories of~\eqref{eq:closed_loop} satisfy
  \begin{align*}
    \norm{
      \begin{bmatrix}
        \dist{\ut}{U_\star}
        \\
        \dist{\zt}{\zss(\ut)}
      \end{bmatrix}} 
      \leq C \rho^\iter
      \norm{
        \begin{bmatrix}
          \dist{u\ud0}{U_\star}
          \\
          \dist{\z\ud0}{\zss(u\ud0)}
        \end{bmatrix}},
  \end{align*}
  for all $(u\ud0,\z\ud0) \in \R^{\n} \times \R^{\m}$ and $\iter \in \N$.
\end{proposition}
\begin{proof}
  The proof of the first result is a direct application of Theorem~\ref{th:contraction}.
  Indeed, by inspecting the closed-loop system~\eqref{eq:closed_loop},  we see that it can be written as an instance of the feedback-interconnected system~\eqref{eq:SP_system}.
  Specifically, we can identify
  \begin{subequations}
    \begin{align}
      x &:= u
      \\
      \slow(x,z) &:= -\step\nabla\zss(x)\nabla\cost(z)\label{eq:slow_fo}
      \\
      \fast(x,z) &:= \dyn(z,x).
    \end{align}
  \end{subequations}
  Hence, to apply Theorem~\ref{th:contraction}, let us check that all the required assumptions are satisfied.
  First of all, system~\eqref{eq:closed_loop} trivially satisfies Assumption~\ref{ass:forward_invariance} with $\X = \R^{\n}$ and $\Z = \R^{\m}$.
  Then, to check Assumption~\ref{ass:set_fix_parametrized}, we note that, by assumption, the point $\zss(u)$ is an equilibrium of subsystem~\eqref{eq:closed_loop_plant} for all $u \in \R^\n$ and, thus, we have
  \begin{align}\label{eq:zeq_fo} 
    \Zeq(u) = \{z \in \R^{\m} \mid z = \zss(u)\}.
  \end{align}
  In particular, the Lipschitz continuity condition~\eqref{eq:lipschitz_eq} in Assumption~\ref{ass:set_fix_parametrized} is verified by the fact that the equilibrium function $\zss$ is $\lipeq$-Lipschitz continuous by hypothesis.
  Analogously, we remark that, for fixed $u$, contraction holds true by hypothesis so that paracontractivity of the fast operator $\fast$ (i.e., Assumption~\ref{ass:fast_paracontractive}) is verified too.
  Now, we recall that the reduced system associated to~\eqref{eq:closed_loop} must be obtained by studying the slow dynamics~\eqref{eq:closed_loop_control_input} with $\pr = 1$ and $\zt \in \Zeq(\ut)$ for all $\iter \in \N$, see Assumption~\ref{ass:reduced_operator}.
  In detail, by observing the slow operator definition in~\eqref{eq:slow_fo} and the definition of the fast equilibrium set in~\eqref{eq:zeq_fo}, the reduced system associated to~\eqref{eq:closed_loop} explicitly corresponds to
  \begin{align}\label{eq:reduced_feedback_opt}
    \utp = \ut - \step\nabla\cost(\zss(\ut)),
  \end{align}
  which means that system~\eqref{eq:closed_loop} satisfies Assumption~\ref{ass:reduced_operator} with
  \begin{align}
    \red(x) = -\step\nabla\cost(\zss(x)).
  \end{align}
  That is, the reduced system exactly corresponds to the gradient descent method applied to the reduced cost function $\cost(\zss(u))$.
  Therefore, since $\cost(\zss(\cdot))$ is $\str$-strongly convex and has an $\lip$-Lipschitz continuous gradient by hypothesis, we know that, for all $\step \in (0,2/(\str +\lip)]$, system~\eqref{eq:reduced_feedback_opt} satisfies Assumption~\ref{ass:slow_contractive} with $\cs = 1 - \sqrt{1 - \step 2\str\lip/(\str+\lip)}$ (see, e.g.,~\cite{ryu2016primer}).
  Finally, the Lipschitz conditions in Assumption~\ref{ass:lip} hold with $\lips = \step\lipeq\lip_\cost$ and $\lipr = \step\lip$.
  Therefore, all the assumptions of Theorem~\ref{th:contraction} are satisfied, and the claim follows by direct application of the theorem.

  The second result instead relies on Theorem~\ref{th:linear_convergence_without_contraction}.
  The only difference from the previous case is that the reduced system~\eqref{eq:reduced_feedback_opt}
  is no longer contractive.
  Nevertheless, the convexity and quadratic growth of $\cost(\zss(\cdot))$
  guarantee that the reduced operator satisfies
  Assumption~\ref{ass:averaged_and_metric_subregular} with
  $\alpha = \step\lip/2$ and $\reg = 2/(\step\str)$.
  Thus, by invoking Lemma~\ref{lemma:linear_convergence_reduced_operator} using these parameters, we obtain the linear convergence property expressed in~\eqref{eq:linear_convergence_without_contraction} with
  \begin{align*}
    \cs = 1 - \sqrt{1 - \frac{\step\str^2(2 - \step\lip)}{4\lip}}.
  \end{align*}
  Hence, all the assumptions of Theorem~\ref{th:linear_convergence_without_contraction} are satisfied, and the
  claim follows by direct application of the theorem.
\end{proof}

\subsection{Stochastic Feedback Optimization}

Inspired by~\cite{carnevale2025almost}, we next consider a stochastic version of this framework.
In particular, rather than considering a deterministic plant, we consider a stochastic plant described by
\begin{align}\label{eq:stochastic_plant}
  \ztp = \dyn(\zt,\ut,\rt),
\end{align}
where $\dyn: \R^{m} \times \R^{n} \times \R^{\nr} \to \R^{m}$ describes the plant dynamics and $\{\rt\}_{\iter\in\N}$ is an i.i.d. sequence of $\cD$-valued random variables with $\cD \subseteq \R^{\nr}$.
Accordingly, the steady-state function $\zss$ now satisfies
\begin{align*}
  \zss(u) = \dyn(\zss(u),u,\rand),
\end{align*}
for all $u \in \R^{\n}$ and $\rand \in \cD$.
We consider the same optimization problem~\eqref{eq:problem} and, thus, the same gradient-based controller~\eqref{eq:closed_loop_control_input} as in the deterministic case.
Hence, the resulting closed-loop system is given by
\begin{subequations}\label{eq:closed_loop_stochastic}
  \begin{align}
    \utp &= \ut - \pr\step\nabla\zss(\ut)\nabla\cost(\zt)
    \label{eq:closed_loop_control_input_stochastic}
    \\
    \ztp &= \dyn(\zt,\ut,\rt).
    \label{eq:closed_loop_plant_stochastic}
  \end{align}
\end{subequations}
The next proposition establishes almost sure convergence of the closed-loop system~\eqref{eq:closed_loop_stochastic} by leveraging Corollary~\ref{cor:linear_convergence_stoch} and Theorem~\ref{th:linear_convergence_without_contraction_stoch} in the strongly convex and convex cases, respectively.
\begin{proposition}\label{prop:feedback_opt_stoch}
  Consider system~\eqref{eq:closed_loop_stochastic}.
  Assume that $\cost(\zss(\cdot))$ is differentiable, $\str$-strongly convex, and has an $\lip$-Lipschitz continuous gradient for some $\str, \lip > 0$.
  Moreover, assume that $\zss$ is differentiable and there exists $\cf \in (0,1)$ such that
  \begin{align}\label{eq:contraction_stochastic}
    \E[\dist{\dyn(\z,u,\rt)}{\zss(u)}] \leq \cf\dist{\z}{\zss(u)},
  \end{align}
  for all $(\z,u) \in \R^{\m} \times \R^{\n}$ and $\iter \in \N$. %
  Further, let $\nabla\cost$ and $\zss$ be $\lip_\cost$- and $\lipeq$-Lipschitz continuous for some $\lip_\cost, \lipeq > 0$, respectively.
  Let 
  \begin{align}
    \bar{\pr} := \frac{(1 - \sqrt{1 - \step 2\str\lip/(\str + \lip)})(1 -\cf)}{\step\lipeq^2\lip_\cost((1 - \sqrt{1 - \step 2\str\lip/(\str + \lip)}) + \step\lip)}.\label{eq:bpr_sc_stoch}
  \end{align}
  Then, for all $\step \in (0,2/(\str + \lip)]$ and $\pr \in (0,\min\{\bar{\pr},1\})$, there exist $C > 0$ and $\rho \in (0,1)$ such that the trajectories of~\eqref{eq:closed_loop_stochastic} satisfy
  \begin{align*}
      \norm{\begin{bmatrix}
        \E[\dist{\ut}{\ustar}]
        \\
        \E[\dist{\zt}{\zss(\ut)}]
      \end{bmatrix}}
      &\leq C \rho^\iter
      \norm{\begin{bmatrix}
          \dist{u\ud0}{\ustar}
          \\
          \dist{\z\ud0}{\zss(u\ud0)}
        \end{bmatrix}},
    \end{align*}
  for all $(u\ud0,\z\ud0) \in \R^{\n} \times \R^{\m}$ and $\iter \in \N$, where $\ustar \in \R^{\n}$ is the unique minimizer of $\cost(\zss(\cdot))$.
  Next, suppose that all the preceding assumptions hold, except that $\cost(\zss(\cdot))$ is only convex and satisfies the quadratic growth condition
  \begin{align}\label{eq:quadratic_growth}
    \cost(\zss(u)) - \costzstar \geq \frac{\str}{2} \dist{u}{U_\star}^2,
  \end{align}
  for some $\str > 0$, where $U_\star$ denotes the set of minimizers of $\cost(\zss(\cdot))$, and $\costzstar$ is the minimum value of $\cost(\zss(\cdot))$.
  Let 
  \begin{align}\label{eq:bpr_c_stoch}
    \bar{\pr} := \frac{\left(1 - \sqrt{1 - \frac{\step\str^2(2 - \step\lip)}{4\lip}}\right)(1 -\cf)}{\step\lipeq^2\lip_\cost\left(\left(1 - \sqrt{1 - \frac{\step\str^2(2 - \step\lip)}{4\lip}}\right) + \step\lip\right)}.
  \end{align}
  Then, for all $\step \in (0,2/\lip)$ and $\pr \in (0,\min\{\bar{\pr},1\})$, there exist $C > 0$ and $\rho \in (0,1)$ such that the trajectories of~\eqref{eq:closed_loop_stochastic} satisfy
  \begin{align*}
      \norm{\begin{bmatrix}
        \E[\dist{\ut}{U_\star}]
        \\
        \E[\dist{\zt}{\zss(\ut)}]
      \end{bmatrix}}
      &\leq C \rho^\iter
        \norm{\begin{bmatrix}
           \dist{u\ud0}{U_\star}
          \\
          \dist{\z\ud0}{\zss(u\ud0)}
        \end{bmatrix}}
      ,
  \end{align*}
  for all $(u\ud0,\z\ud0) \in \R^{\n} \times \R^{\m}$.
\end{proposition}
\begin{proof}
  The proof is the same as that of Proposition~\ref{prop:feedback_opt} with the only difference that we apply Corollary~\ref{cor:linear_convergence_stoch} and Theorem~\ref{th:linear_convergence_without_contraction_stoch} rather than Theorem~\ref{th:contraction} and Theorem~\ref{th:linear_convergence_without_contraction} because of the stochastic nature of the plant~\eqref{eq:closed_loop_plant_stochastic}.
\end{proof}

\subsection{Numerical Simulations}

In this section, we test the closed-loop system~\eqref{eq:closed_loop} with $\pr$ tuned using the theoretical bound provided by Proposition~\ref{prop:feedback_opt}.
First, we consider $\n = \m = 3$ and a linear deterministic plant 
\begin{align}
  \ztp = \zt + \kappa(\ut - \zt),\label{eq:linear_plant}
\end{align}
where $\kappa$ is sampled uniformly from the interval $[0.1,0.9]$.
Hence, the steady-state function of system~\eqref{eq:linear_plant} is $\zss(u) = u$ and the contraction rate is $\cf = 1 - \kappa$.
We generate the cost function $\cost$ as the quadratic function
\begin{align}
  \cost(z) = \frac{1}{2}z\T Qz + z\T b,\label{eq:quadratic_cost_feedback_opt}
\end{align}
where the matrix $Q = Q\T \in \R^{3 \times 3}$ is randomly generated so that its eigenvalues lie in $[1,5]$, while each component of $b \in \R^3$ is sampled uniformly from the interval $[-10,20]$.
In all simulations, system~\eqref{eq:closed_loop} is initialized at the same random state, sampled from a Gaussian distribution with zero mean and covariance matrix $100 \cdot I_6$.
Denoting by $\ustar \in \R^3$ the unique minimizer of~\eqref{eq:quadratic_cost_feedback_opt}, Fig.~\ref{fig:feedback_optimization} shows the evolution, over the iterations $\iter$, of the error $\norm{\ut - \ustar}$ for the reduced system (i.e., subsystem~\eqref{eq:closed_loop_control_input} with $\pr = 1$ and $\zt = \zss(\ut)$ for all $\iter \in \N$) and the closed-loop system~\eqref{eq:closed_loop} with $\pr = 1$ (i.e., without timescale separation) and with $\pr = 0.999\bar{\pr}$ tuned according to~\eqref{eq:bpr_sc} in Proposition~\ref{prop:feedback_opt}.
All three cases are tested with $\step = 2/(\str +\lip)$, where $\str > 0$ and $\lip >0 $ are the smallest and largest eigenvalues of $Q$ in~\eqref{eq:quadratic_cost_feedback_opt}, respectively.
Fig.~\ref{fig:feedback_optimization} shows that the closed-loop system~\eqref{eq:closed_loop} converges to the unique minimizer of problem~\eqref{eq:quadratic_cost_feedback_opt} at a linear rate when the parameter $\pr$ is tuned according to~\eqref{eq:bpr_sc} in Proposition~\ref{prop:feedback_opt}.
Moreover, Fig.~\ref{fig:feedback_optimization} also highlights the \emph{necessity} of timescale separation: in the case with $\pr = 1$, even when using a stepsize $\step$ that guarantees convergence of the reduced system (i.e., the gradient method) both theoretically ($\step = 2/(\str + \lip)$) and numerically (see Fig.~\ref{fig:feedback_optimization}), the closed-loop system exhibits unstable behavior.
\begin{figure}[H]
  \centering
  \includegraphics[scale=1]{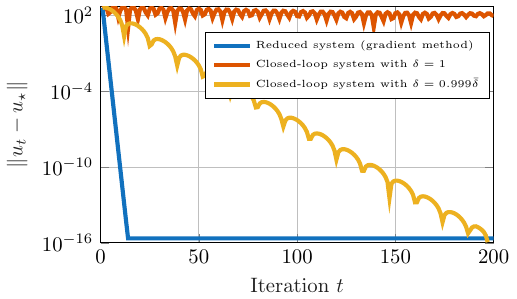}
  \caption{Error $\norm{\ut - \ustar}$ achieved in the strongly convex case by the closed-loop system~\eqref{eq:closed_loop} with different values of $\pr$, together with the reduced system~\eqref{eq:reduced_feedback_opt}. All three cases are tested with $\step = 2/(\str + \lip)$.}
  \label{fig:feedback_optimization}
\end{figure}
To test~\eqref{eq:bpr_c} as well, we consider the same setup as before, but generate the cost function $\cost$ as a quadratic function with a positive semidefinite matrix $Q = Q^\top \in \R^{3 \times 3}$ having one zero eigenvalue, while the remaining two eigenvalues are sampled from the interval $[1,5]$.
Hence, the resulting cost function $\cost(\zss(\cdot))$ is convex (but not strongly convex) and satisfies the quadratic growth condition~\eqref{eq:quadratic_growth} with $\str > 0$ corresponding to the smallest nonzero eigenvalue of $Q$.
Moreover, we modify the generation of $b$ by ensuring that $b$ is orthogonal to the kernel of $Q$, which guarantees the existence of nonunique minimizers of $\cost(\zss(\cdot))$.
As in the previous experiment, we test the reduced system (i.e., subsystem~\eqref{eq:closed_loop_control_input} with $\pr = 1$ and $\zt = \zss(\ut)$ for all $\iter \in \N$) and the closed-loop system~\eqref{eq:closed_loop} with $\pr = 1$ (i.e., without timescale separation) and with $\pr = 0.999\bar{\pr}$ tuned according to~\eqref{eq:bpr_c} in Proposition~\ref{prop:feedback_opt}.
All three cases are tested with $\step = 1/\lip$, where $\lip >0$ corresponds to the largest eigenvalue of $Q$.
This experiment again confirms the effectiveness of the proposed tuning strategy and the necessity of timescale separation to guarantee convergence of the closed-loop system~\eqref{eq:closed_loop}.
\begin{figure}[H]
  \centering
  \includegraphics[scale=1]{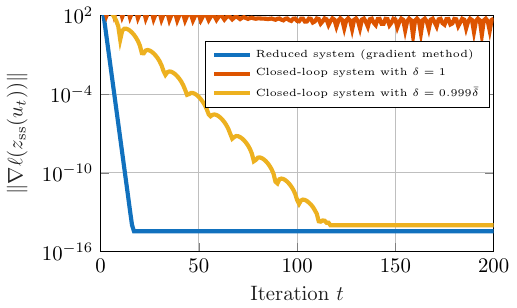}
  \caption{Optimality error $\norm{\nabla\cost(\zss(\ut))}$ achieved in the convex case by the closed-loop system~\eqref{eq:closed_loop} with different values of $\pr$, together with the reduced system~\eqref{eq:reduced_feedback_opt}. 
  All three cases are tested with $\step = 1/\lip$.}
  \label{fig:feedback_optimization_without_contraction}
\end{figure}
Then, we test the stochastic result in Proposition~\ref{prop:feedback_opt_stoch}.
We consider the same setup as before, except that the plant is now modeled as the stochastic linear system
\begin{align}
  \ztp = \zt + \rt(\ut - \zt), 
  \label{eq:stochastic_linear_plant}
\end{align}
where $\{\rt\}_{\iter\in\N}$ is an independent and identically distributed
sequence of Bernoulli random variables.
In particular, we sample the Bernoulli parameter $\E[\rt]$ uniformly from the interval $[0,1]$, which ensures that~\eqref{eq:contraction_stochastic} is satisfied with $\cf = 1 - \E[\rt]$.
We consider the same strongly convex and convex setups as above, with the same tests, parameters, and using the tuning rules provided by~\eqref{eq:bpr_sc_stoch} and~\eqref{eq:bpr_c_stoch} in Proposition~\ref{prop:feedback_opt_stoch}, respectively.
Fig.~\ref{fig:feedback_optimization_stochastic} shows the results of the tests performed in the strongly convex case, while Fig.~\ref{fig:feedback_optimization_without_contraction_stochastic} shows the results of the tests performed in the convex case.
As before, the results confirm the effectiveness of the proposed tuning strategy and the necessity of timescale separation to guarantee convergence of the closed-loop system~\eqref{eq:closed_loop_stochastic}.
\begin{figure}[H]
  \centering
  \includegraphics[scale=1]{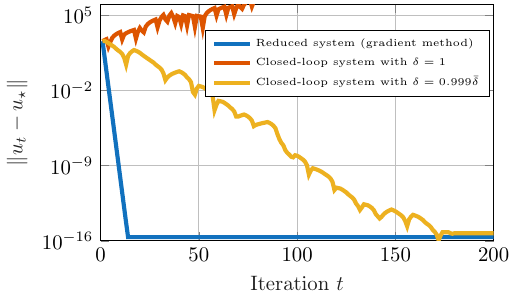}
  \caption{Error $\norm{\ut - u_\star}$ achieved in the strongly convex case by the closed-loop system~\eqref{eq:closed_loop_stochastic} with different values of $\pr$, together with the reduced system~\eqref{eq:reduced_feedback_opt}. All three cases are tested with $\step = 2/(\str + \lip)$.}
  \label{fig:feedback_optimization_stochastic}
\end{figure}
\begin{figure}[H]
  \centering
  \includegraphics[scale=1]{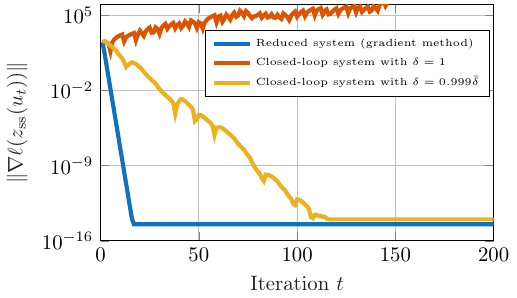}
  \caption{Optimality error $\norm{\nabla\cost(\zss(\ut))}$ achieved in the convex case by the closed-loop system~\eqref{eq:closed_loop_stochastic} with different values of $\pr$, together with the reduced system~\eqref{eq:reduced_feedback_opt}. 
  All three cases are tested with $\step = 1/\lip$.}
  \label{fig:feedback_optimization_without_contraction_stochastic}
\end{figure}

\section{Conclusions}
\label{sec:conclusions}

In this paper, we developed an operator-theoretic framework for
timescale separation in interconnected discrete-time systems.
By describing the fast dynamics through convergence toward a slow-state-dependent fixed-point set and the slow dynamics through an associated reduced operator, we obtained explicit and readily checkable bounds on the timescale parameter to guarantee convergence of the full interconnection. 
These bounds are expressed in terms of the operator constants appearing in the assumptions.
The deterministic analysis covers both contractive reduced operators
and the broader class of averaged and metrically subregular operators,
for which linear convergence to the corresponding fixed-point set was
established.
We also extended the framework to stochastic operators and provided
conditions ensuring almost sure convergence of the interconnected
dynamics.
The application to deterministic and stochastic feedback optimization
illustrated how the proposed results can be used to derive systematic
tuning rules and to quantify the separation required between plant and
optimization dynamics.

\appendix

\subsection{Proof of Theorem~\ref{th:contraction}}
\label{sec:proof_th_contraction}

Let us consider an arbitrary pair $(x,z) \in \X \times \Z$ and let $\xp := \x + \pr\slow(x,z)$ and $\zp := \fast(x,z)$ be the next state of system~\eqref{eq:SP_system} starting from $(x,z)$.
Then, we can write
  \begin{align}
    \dists{\xp}{\xstar}
    &= \dists{\x + \pr\slow(\x,\z)}{\xstar}
    \notag\\
    &\stackrel{(a)}{\leq}
    \dists{\x + \pr\red(\x)}{\xstar} + \pr\dists{\slow(\x,\bzx)}{\slow(\x,\z)}
    \notag\\
    &\stackrel{(b)}{\leq}
    (1-\delta)\dists{\x}{\xstar} + \pr\dists{\x + \red(\x)}{\xstar} 
    + \pr\dists{\slow(\x,\bzx)}{\slow(\x,\z)}
    \notag\\
    &\stackrel{(c)}{\leq}
    (1-\pr)\dists{\x}{\xstar}+ \pr\ocs\dists{\x}{\xstar} 
    + \pr\lips\distf{\z}{\bzx}
    \notag\\
    &\stackrel{(d)}{=}
    (1-\pr\cs)\dists{\x}{\xstar}+  \pr\lips\distf{\z}{\Zeq(\x)},\label{eq:xtp_xstar}
  \end{align}
  where in $(a)$ we choose $\bzx \in \arg\inf_{y \in \Zeq(\x)}\distf{\z}{y}$, add and subtract $\pr\red(\x) = \pr\slow(\x,\bzx)$ (cf. Assumption~\ref{ass:reduced_operator}) within $\dists{\x + \pr\slow(\x,\z)}{\xstar}$ and apply the triangle inequality, in $(b)$ we use the decompositions $\x = (1-\pr)\x + \pr\x$ and $\xstar = (1-\pr)\xstar + \pr \xstar$, in $(c)$ we apply the triangle inequality, factor out the nonnegative scalars $\pr$ and $(1-\pr)$, 
  ~\eqref{eq:slow_paracontractive} in Assumption~\ref{ass:slow_contractive}, and the $\lips$-Lipschitz continuity of $\slow$ (cf. Assumption~\ref{ass:lip}), while in $(d)$ we use the fact that $\bzx \in \arg\inf_{y \in \Zeq(\x)}\distf{\z}{y}$ by definition and reorganize the other terms.
  We now evaluate the quantity $\distf{\zp}{\Zeq(\xp)}$.
  By using the definition of $\distf{\cdot}{\Zeq(\x)}$ and adding and subtracting $\bzxz \in \arg\inf_{y \in \Zeq(\x)}\distf{\fast(\x,\z)}{y}$, we get
  \begin{align}
    \distf{\zp}{\Zeq(\xp)} 
    &= 
    \inf_{y \in \Zeq(\xp)}\distf{\fast(\x,\z)}{y}
    \notag\\
    &\stackrel{(a)}{\leq} 
    \distf{\fast(\x,\z)}{\bzxz} +  \inf_{y \in \Zeq(\xp)}\distf{\bzxz}{y}
    \notag\\
    &\stackrel{(b)}{\leq} 
    \distf{\fast(\x,\z)}{\Zeq(\x)} 
    + \lipeq\dists{\xp}{\x}
    \notag\\
    &\stackrel{(c)}{\leq} 
    \cf\distf{\z}{\Zeq(\x)} 
    + \pr\lipeq\norms{\slow(\x,\z)}
    \notag\\
    &\stackrel{(d)}{\leq} 
    \cf\distf{\z}{\Zeq(\x)}
    + \pr\lipeq\dists{\slow(\x,\z)}{\slow(\x,\bzx)}
    \notag\\
    &\hspace{.4cm}
    + \pr\lipeq\dists{\red(\x)}{\red(\xstar)}
    \notag\\
    &\stackrel{(e)}{\leq} 
    \cf\distf{\z}{\Zeq(\x)}
    + \pr\lipeq\lips\distf{\z}{\bzx}
    \notag\\
    &\hspace{.4cm}
    + \pr\lipeq\lipr\dists{\x}{\xstar}
    \notag\\
    &\stackrel{(f)}{=} 
    \cf\distf{\z}{\Zeq(\x)}
    + \pr\lipeq\lips\distf{\z}{\Zeq(\x)}
    + \pr\lipeq\lipr\dists{\x}{\xstar}
    ,\label{eq:ztp_fix_xt}
  \end{align}
  where in $(a)$ we use the triangle inequality, in $(b)$ we use the fact that $\bzxz \in \arg\inf_{y \in \Zeq(\x)}\distf{\fast(\x,\z)}{y}$, the definition of $\distf{\fast(\x,\z)}{\Zeq(\x)}$, and the $\lipeq$-Lipschitz continuity bound~\eqref{eq:lipschitz_eq} (cf. Assumption~\ref{ass:set_fix_parametrized}), in $(c)$ we apply the $\cf$-contractivity of $\fast$ (cf. Assumption~\ref{ass:fast_paracontractive}) and use the slow evolution~\eqref{eq:SP_system_slow}, in $(d)$ we choose $\bzx \in \arg\inf_{y \in \Zeq(\x)}\distf{\z}{y}$ and add and subtract within the last norm the term $\slow(\x,\bzx) = \red(\x)$, apply the triangle inequality and use the fact that $\red(\xstar) = 0$ (cf. Assumption~\ref{ass:slow_contractive}), in $(e)$ we use the Lipschitz continuity of $\slow$ and $\red$ (cf. Assumption~\ref{ass:lip}), while in $(f)$ we use the fact that $\bzx \in \arg\inf_{y \in \Zeq(\x)}\distf{\z}{y}$ by definition.
  Hence, by collecting~\eqref{eq:xtp_xstar} and~\eqref{eq:ztp_fix_xt}, we can write
  \begin{align}
    \begin{bmatrix}
      \dists{\xp}{\xstar} 
      \\
      \distf{\zp}{\Zeq(\xp)}
    \end{bmatrix} 
    \leq 
    M(\pr)
    \begin{bmatrix}
      \dists{\x}{\xstar}
      \\
      \distf{\z}{\Zeq(\x)}
    \end{bmatrix},
    \label{eq:norm_evolution}
  \end{align}
  where the matrix $M(\pr) \in \R^{2 \times 2}$ is defined as 
  \begin{align}
    M(\pr) &:= 
    \begin{bmatrix}
      1 - \pr\cs& \pr\lips  
      \\
      \pr\lipeq\lipr& \cf + \pr\lipeq\lips 
    \end{bmatrix}.
    \label{eq:M_pr}
  \end{align}
  Since $M(\pr)$ is a nonnegative matrix (i.e., all its entries are nonnegative), it is Schur if and only if $I_2 - M(\pr)$ is a nonsingular $M$-matrix~\cite{plemmons1977m}.
  By definition of $M(\pr)$ (cf.~\eqref{eq:M_pr}), we write
  \begin{align*}
    I_2 - M(\pr) &= 
    \begin{bmatrix}
      \pr\cs& -\pr\lips  
      \\
      -\pr\lipeq\lipr& 1 - \cf - \pr\lipeq\lips 
    \end{bmatrix}.
  \end{align*}
  Moreover, by~\cite[Thm.~1]{plemmons1977m}, $I_2 - M(\pr)$ is a nonsingular $M$-matrix if and only if the principal minors of $I_2 - M(\pr)$ are positive.
  In turn, since $\pr\cs > 0$ holds trivially for all $\pr > 0$, this is equivalent to
  \begin{align}
    &\pr\cs(1 - \cf - \pr\lipeq\lips) - \pr^2\lipeq\lipr\lips > 0
    \iff 
    \pr < \frac{\cs(1 - \cf)}{\lipeq\lips(\cs + \lipr)}.\label{eq:condition}
  \end{align}
  The proof then follows by noting that the right-hand side of~\eqref{eq:condition} corresponds to the definition of $\bar{\pr}$ in~\eqref{eq:bar_pr} and by applying~\eqref{eq:norm_evolution} recursively to obtain the linear rate~\eqref{eq:linear_rate}.

\subsection{Proof of Lemma~\ref{lemma:linear_convergence_reduced_operator}}
\label{sec:proof_lemma_linear_convergence_reduced_operator}

Let us introduce the operator $\op: \X \to \R^{\n}$ defined as 
\begin{align}\label{eq:G_definition} 
  \op(\x) := \x + \frac{1}{\alpha}\red(\x),
\end{align} 
which allows us to equivalently rewrite the reduced operator $\x + \red(\x)$ as
\begin{align}\label{eq:averaged_operator}
  \x + \red(\x) = (1-\alpha)\x + \alpha\op(\x).
\end{align}
Since $\x + \red(\x)$ is $\alpha$-averaged in $\X$ (cf. Assumption~\ref{ass:averaged_and_metric_subregular}), the operator $\op$ is nonexpansive (cf.~\eqref{eq:non_expansive}).
Now, by using~\eqref{eq:averaged_operator}, we can write
\begin{align}
  \dists{\x + \red(\x)}{\bxx}^2
  &=\norms{(1-\alpha)\x + \alpha\op(\x)-\bxx}^2
  \notag\\
  &\stackrel{(a)}{=} 
  \norms{(1 - \alpha)(\x - \bxx) + \alpha(\op(\x) - \bxx)}^2
  \notag\\
  &\stackrel{(b)}{=} 
  (1 - \alpha)\norms{\x - \bxx}^2 + \alpha\norms{\op(\x) - \bxx}^2 
  - \alpha(1 - \alpha)\norms{\x -\bxx - \op(\x) + \bxx}^2
  \notag\\
  &\stackrel{(c)}{=} 
  (1 - \alpha)\norms{\x -\bxx}^2 + \alpha\norms{\op(\x) - \op(\bxx)}^2 
  - \alpha(1 - \alpha)\norms{\x - \op(\x)}^2
  \notag\\
  &\stackrel{(d)}{\leq} 
  \norms{\x -\bxx}^2 - \frac{1 - \alpha}{\alpha}\norms{\red(\x)}^2,\label{eq:last}
\end{align}
where in $(a)$ we use the decomposition $\bxx = (1 - \alpha)\bxx + \alpha \bxx$, in $(b)$ we use the standard algebraic property $\norms{(1 - \alpha)a + \alpha b}^2 = (1-\alpha)\norms{a}^2 + \alpha\norms{b}^2 - \alpha(1 - \alpha)\norms{a - b}^2$ which holds for all $\alpha \in \R$ and $a,b \in \R^{\n}$, in $(c)$ we use the fact that $\op(\bxx) = \bxx$ since $\bxx \in \Xeq$ and simplify the terms in the last norm, while in $(d)$ we use the fact that $\op(\x)$ is nonexpansive and that $\op(\x) - \x = \frac{1}{\alpha}\red(\x)$ (cf.~\eqref{eq:G_definition}).
The proof follows by combining~\eqref{eq:last} with the fact that $\x + \red(\x)$ is $\reg$-metric subregular (cf. Assumption~\ref{ass:averaged_and_metric_subregular}) and the fact that $\bxx \in \arg\inf_{y \in \Xeq}\dists{\x}{y}$.

\subsection{Proof of Theorem~\ref{th:linear_convergence_without_contraction}}
\label{sec:proof_th_linear_convergence_without_contraction}

Let us consider an arbitrary pair $(\x,\z) \in \X \times \Z$ and let $\xp := \x + \pr\slow(\x,\z)$ and $\zp := \fast(\x,\z)$ be the next state of system~\eqref{eq:SP_system} starting from $(\x,\z)$.
  Then, by taking $\bxx \in \arg\inf_{y \in \Xeq}\dists{\x}{y}$ and considering the definition of $\dists{\x}{\Xeq}$, we can write
  \begin{align}
    \dists{\xp}{\Xeq} 
    &\leq \dists{\x + \pr\slow(\x,\z)}{\bxx}
    \notag\\
    &\stackrel{(a)}{\leq}
    \norms{\x + \pr\red(\x) - \bxx} + \pr\norms{\slow(\x,\bzx)-\slow(\x,\z)}
    \notag\\
    &\stackrel{(b)}{=}
    \norms{(1-\pr)(\x - \bxx) + \pr(\x + \red(\x) - \bxx)} 
    + \pr\norms{\slow(\x,\bzx)-\slow(\x,\z)}
    \notag\\
    &\stackrel{(c)}{\leq}
    (1-\pr)\norms{\x - \bxx} + \pr\norms{\x + \red(\x) - \bxx} 
    + \pr\norms{\slow(\x,\bzx)-\slow(\x,\z)}
    \notag\\
    &\stackrel{(d)}{\leq}
    (1-\pr)\norms{\x - \bxx}+ \pr\ocs\norms{\x - \bxx} 
    + \pr\lips\normf{\z - \bzx}
    \notag\\
    &\stackrel{(e)}{=}
    (1-\pr\cs)\dists{\x}{\Xeq}+  \pr\lips\distf{\z}{\Zeq(\x)},\label{eq:xtp_Xeq}
  \end{align}
  where in $(a)$ we choose $\bzx \in \arg\inf_{y \in \Zeq(\x)}\distf{\z}{y}$, add and subtract $\pr\red(\x) = \pr\slow(\x,\bzx)$ (cf. Assumption~\ref{ass:reduced_operator}) within the norm and apply the triangle inequality, in $(b)$ we use the decomposition $\x = (1-\pr)\x + \pr\x$ and $\bxx = (1-\pr)\bxx + \pr \bxx$, in $(c)$ we use the triangle inequality and factor out the nonnegative scalars $\pr$ and $(1-\pr)$, 
  in $(d)$ we apply Lemma~\ref{lemma:linear_convergence_reduced_operator} and the $\lips$-Lipschitz continuity of $\slow$ (cf. Assumption~\ref{ass:lip}), while in $(e)$ we use the fact that $\bzx \in \arg\inf_{y \in \Zeq(\x)}\distf{\z}{y}$ by definition, reorganize the other terms, and use the definition of $\dists{\cdot}{\Xeq}$.
  We now evaluate the quantity $\distf{\zp}{\Zeq(\xp)}$ along the trajectories of~\eqref{eq:SP_system}, use the definition of $\distf{\cdot}{\Zeq(\x)}$, and add and subtract $\bzxz \in \arg\inf_{y \in \Zeq(\x)}\distf{\fast(\x,\z)}{y}$ to obtain
  \begin{align}
    \distf{\zp}{\Zeq(\xp)} 
    &= 
    \inf_{y \in \Zeq(\xp)}\distf{\fast(\x,\z) -\bzxz + \bzxz}{y}
    \notag\\
    &\stackrel{(a)}{\leq} 
    \normf{\fast(\x,\z) -\bzxz} +  \inf_{y \in \Zeq(\xp)}\normf{\bzxz - y}
    \notag\\
    &\stackrel{(b)}{\leq} 
    \distf{\fast(\x,\z)}{\Zeq(\x)} + \lipeq\norms{\xp - \x}
    \notag\\
    &\stackrel{(c)}{\leq} 
    \cf\distf{\z}{\Zeq(\x)} 
    + \pr\lipeq\norms{\slow(\x,\z)}
    \notag\\
    &\stackrel{(d)}{\leq} 
    \cf\distf{\z}{\Zeq(\x)}
    + \pr\lipeq\dists{\slow(\x,\z)}{\slow(\x,\bzx)} 
    + \pr\lipeq\norms{\red(\x)-\red(\bxx)}
    \notag\\
    &\stackrel{(e)}{\leq} 
    \cf\distf{\z}{\Zeq(\x)}
    + \pr\lipeq\lips\distf{\z}{\bzx} 
    + \pr\lipeq\lipr\dists{\x}{\bxx}
    \notag\\
    &\stackrel{(f)}{=} 
    \cf\distf{\z}{\Zeq(\x)}
    + \pr\lipeq\lips\distf{\z}{\Zeq(\x)}
    + \pr\lipeq\lipr\dists{\x}{\Xeq}
    ,\label{eq:ztp_fix_xt_averaged_case}
  \end{align}
  where in $(a)$ we use the triangle inequality, in $(b)$ we use the fact that $\bzxz \in \arg\inf_{y \in \Zeq(\x)}\distf{\fast(\x,\z)}{y}$, the definition of $\distf{\fast(\x,\z)}{\Zeq(\x)}$, and the $\lipeq$-Lipschitz continuity bound~\eqref{eq:lipschitz_eq} (cf. Assumption~\ref{ass:set_fix_parametrized}), in $(c)$ we apply the $\cf$-contractivity of $\fast$ (cf. Assumption~\ref{ass:fast_paracontractive}) and use the slow evolution~\eqref{eq:SP_system_slow}, in $(d)$ we choose 
  $\bzx \in \arg\inf_{y \in \Zeq(\x)}\distf{\z}{y}$, $\bxx \in \arg\inf_{y \in \Xeq}\dists{\x}{y}$, add and subtract within the last norm the term $\slow(\x,\bzx) = \red(\x)$, apply the triangle inequality and use the fact that $\red(\bxx) = 0$ by construction, in $(e)$ we use the Lipschitz continuity of $\slow$ and $\red$ (cf. Assumption~\ref{ass:lip}), while in $(f)$ we use the fact that $\bzx \in \arg\inf_{y \in \Zeq(\x)}\distf{\z}{y}$ and $\bxx \in \arg\inf_{y \in \Xeq}\dists{\x}{y}$ by definition.
  Hence, by collecting~\eqref{eq:xtp_Xeq} and~\eqref{eq:ztp_fix_xt_averaged_case}, we can write 
  \begin{align}
    \begin{bmatrix}
      \dists{\xp}{\Xeq}
      \\
      \distf{\zp}{\Zeq(\xp)}
    \end{bmatrix} 
    \leq 
                  M(\pr)
    \begin{bmatrix}
      \dists{\x}{\Xeq}
      \\
      \distf{\z}{\Zeq(\x)}
    \end{bmatrix},
    \label{eq:norm_evolution_averaged_case}
  \end{align}
  where the matrix $M(\pr) \in \R^{2 \times 2}$ is defined as in~\eqref{eq:M_pr}, although now the parameter $\cs$ has a different meaning (see Lemma~\ref{lemma:linear_convergence_reduced_operator}) from that in~\eqref{eq:norm_evolution} (see Assumption~\ref{ass:slow_contractive}).
  However, we remark that in both cases $\cs \in (0,1)$ and, thus, the matrix appearing in~\eqref{eq:norm_evolution_averaged_case} has the same properties discussed in the proof of Theorem~\ref{th:contraction}.
  Hence, the proof follows by repeating the same arguments as in the proof of Theorem~\ref{th:contraction}, see Appendix~\ref{sec:proof_th_contraction}.

  \subsection{Proof of Lemma~\ref{lemma:linear_convergence_reduced_operator_stoch}}
  \label{sec:proof_lemma_linear_convergence_reduced_operator_stoch}

  Let us introduce the operator $\op: \X \times \cD \to \R^{\n}$ defined as 
  \begin{align}\label{eq:G_definition_stoch} 
    \op(\x,\rand) := \x + \frac{1}{\alpha}\red(\x,\rand),
  \end{align} 
  which allows us to equivalently rewrite the reduced operator $\x + \red(\x,\rand)$ as
  \begin{align}\label{eq:averaged_operator_stoch}
    \x + \red(\x,\rand) = (1-\alpha)\x + \alpha\op(\x,\rand).
  \end{align}
  Since $\x + \red(\x,\rand)$ is stochastically $\alpha$-averaged in $\X$ (cf. Assumption~\ref{ass:averaged_and_metric_subregular_stoch}), the operator $\op(\x,\rand)$ is nonexpansive (cf.~\eqref{eq:non_expansive_stoch}) in a stochastic sense.
  Now, by using~\eqref{eq:averaged_operator_stoch}, we can write
  \begin{align}
    \E\left[\dists{\x + \red(\x,\rand)}{\bxx}^2\right]
    &=\E\left[\norms{(1-\alpha)\x + \alpha\op(\x,r)-\bxx}^2\right]
    \notag\\
    &\stackrel{(a)}{=} 
    \E\left[\norms{(1 - \alpha)(\x - \bxx) + \alpha(\op(\x,\rand) - \bxx)}^2\right]
    \notag\\
    &\stackrel{(b)}{=} 
    (1 - \alpha)\E\left[\norms{\x - \bxx}^2\right] + \alpha\E\left[\norms{\op(\x,\rand) - \bxx}^2\right] 
    - \alpha(1 - \alpha)\E\left[\norms{\x -\bxx - \op(\x,\rand) + \bxx}^2\right]
    \notag\\
    &\stackrel{(c)}{=} 
    (1 - \alpha)\E\left[\norms{\x -\bxx}^2\right] + \alpha\E\left[\norms{\op(\x,\rand) - \op(\bxx,\rand)}^2\right] 
    - \alpha(1 - \alpha)\E\left[\norms{\x - \op(\x,\rand)}^2\right]
    \notag\\
    &\stackrel{(d)}{\leq} 
    \E\left[\norms{\x -\bxx}^2\right] - \frac{1 - \alpha}{\alpha}\E\left[\norms{\red(\x,\rand)}^2\right]
    \notag\\
    &\stackrel{(e)}{\leq} 
    \E\left[\norms{\x -\bxx}^2\right] - \frac{1 - \alpha}{\alpha}\left(\E\left[\norms{\red(\x,\rand)}\right]\right)^2
    \notag\\
    &\stackrel{(f)}{\leq} 
    \left(1 - \tfrac{1 - \alpha}{\alpha\reg^2}\right)\dists{\x}{\bxx}^2,\label{eq:last_stoch}
  \end{align}
  where in $(a)$ we use the decomposition $\bxx = (1 - \alpha)\bxx + \alpha \bxx$, in $(b)$ we use the standard algebraic property $\norms{(1 - \alpha)a + \alpha b}^2 = (1-\alpha)\norms{a}^2 + \alpha\norms{b}^2 - \alpha(1 - \alpha)\norms{a - b}^2$ which holds for all $\alpha \in \R$ and $a,b \in \R^{\n}$, in $(c)$ we use the fact that $\op(\bxx, r) = \bxx$ since $\bxx \in \Xeq$ and simplify the terms in the last norm, in $(d)$ we use the fact that $\op(\x,\rand)$ is stochastically nonexpansive and that $\op(\x,\rand) - \x = \frac{1}{\alpha}\red(\x,\rand)$ (cf.~\eqref{eq:G_definition_stoch}), in $(e)$ we apply Jensen's inequality,
  while in $(f)$ we use the fact that $\x + \red(\x,\rand)$ is $\reg$-metric subregular in a stochastic sense (cf. Assumption~\ref{ass:averaged_and_metric_subregular_stoch}), the fact that $\bxx \in \arg\inf_{y \in \Xeq}\dists{\x}{y}$, and rearrange the terms.
  The proof follows by taking the square root of both sides of~\eqref{eq:last_stoch}.

\subsection{Proof of Theorem~\ref{th:linear_convergence_without_contraction_stoch}}
\label{sec:proof_th_linear_convergence_without_contraction_stoch}

Let us consider an arbitrary pair $(\x,\z,\rand) \in \X \times \Z \times \cD$ and let $\xp := \x + \pr\slow(\x,\z,\rand)$ and $\zp := \fast(\x,\z,\rand)$ be the next state of system~\eqref{eq:SP_system_stochastic} starting from $(\x,\z)$.
We remark that the next state $(\xp,\zp)$ is random since it depends on the random variable $\rand$.
Accordingly, we will handle it through the expectation operator $\E[\cdot]$.
  In particular, by taking $\bxx \in \arg\inf_{y \in \Xeq}\dists{\x}{y}$ and considering the definition of $\dists{\x}{\Xeq}$, we can write
  \begin{align}
    \E\left[\dists{\xp}{\Xeq}\right]
    &\leq \E\left[\dists{\x + \pr\slow(\x,\z,\rand)}{\bxx}\right]
    \notag\\
    &\stackrel{(a)}{\leq}
    \E\left[\norms{\x + \pr\red(\x,\rand) - \bxx}\right] 
     + \pr\E\left[\norms{\slow(\x,\bzx,\rand)-\slow(\x,\z,\rand)}\right]
    \notag\\
    &\stackrel{(b)}{=}
    \E\left[\norms{(1-\pr)(\x - \bxx) + \pr(\x + \red(\x,\rand) - \bxx)}\right]
    + \pr\E\left[\norms{\slow(\x,\bzx,\rand)-\slow(\x,\z,\rand)}\right]
    \notag\\
    &\stackrel{(c)}{\leq}
    (1-\pr)\norms{\x - \bxx} + \pr\E\left[\norms{\x + \red(\x,\rand) - \bxx}\right] 
    + \pr\E\left[\norms{\slow(\x,\bzx,\rand)-\slow(\x,\z,\rand)}\right]
    \notag\\
    &\stackrel{(d)}{\leq}
    (1-\pr)\norms{\x - \bxx}+ \pr\ocs\norms{\x - \bxx} 
    + \pr\lips\normf{\z - \bzx}
    \notag\\
    &\stackrel{(e)}{=}
    (1-\pr\cs)\dists{\x}{\Xeq}+  \pr\lips\distf{\z}{\Zeq(\x)},\label{eq:xtp_Xeq_stoch}
  \end{align}
  where in $(a)$ we choose $\bzx \in \arg\inf_{y \in \Zeq(\x)}\normf{\z - y}$, add and subtract $\pr\red(\x,\rand) = \pr\slow(\x,\bzx,\rand)$ (cf. Assumption~\ref{ass:reduced_operator_stoch}) within the norm and apply the triangle inequality, in $(b)$ we use the decompositions $\x = (1-\pr)\x + \pr\x$ and $\bxx = (1-\pr)\bxx + \pr \bxx$, in $(c)$ we use the triangle inequality and factor out the nonnegative scalars $\pr$ and $(1-\pr)$, 
  in $(d)$ we apply Lemma~\ref{lemma:linear_convergence_reduced_operator_stoch} and the $\lips$-Lipschitz continuity of $\slow$ (cf. Assumption~\ref{ass:lip_stoch}), while in $(e)$ we use the fact that $\bzx \in \arg\inf_{y \in \Zeq(\x)}\distf{\z}{y}$, reorganize the other terms, and use the definition of $\dists{\cdot}{\Xeq}$.
  We now evaluate the increment of $\E[\distf{\zp}{\Zeq(\xp)}]$ along the update of~\eqref{eq:SP_system_stochastic}, use the definition of $\distf{\cdot}{\Zeq(\x)}$, and add $\pm\bzxz \in \arg\inf_{y \in \Zeq(\x)}\distf{\fast(\x,\z,\rand)}{y}$ 
  to get
  \begin{align}
    \E\left[\distf{\zp}{\Zeq(\xp)}\right]
    &= 
    \E\left[\inf_{y \in \Zeq(\xp)}\distf{\fast(\x,\z,\rand) -\bzxz + \bzxz}{y}\right]
    \notag\\
    &\stackrel{(a)}{\leq} 
    \E\!\left[\!\normf{\fast(\x,\z,\rand)-\bzxz} + \inf_{y \in \Zeq(\xp)}\normf{\bzxz - y}\right]
    \notag\\
    &\stackrel{(b)}{\leq} 
    \E\left[\distf{\fast(\x,\z,\rand)}{\Zeq(\x)} + \lipeq\norms{\xp - \x}\right]
    \notag\\
    &\stackrel{(c)}{\leq} 
    \cf\distf{\z}{\Zeq(\x)} 
    + \pr\lipeq\E\left[\norms{\slow(\x,\z,\rand)}\right]
    \notag\\
    &\stackrel{(d)}{\leq} 
    \cf\distf{\z}{\Zeq(\x)}
    + \pr\lipeq\E\left[\dists{\slow(\x,\z,\rand)}{\slow(\x,\bzx,\rand)}\right]
    + \pr\lipeq\E\left[\norms{\red(\x,\rand)-\red(\bxx,\rand)}\right]
    \notag\\
    &\stackrel{(e)}{\leq} 
    \cf\distf{\z}{\Zeq(\x)}
    + \pr\lipeq\lips\distf{\z}{\bzx} 
    + \pr\lipeq\lipr\dists{\x}{\bxx}
    \notag\\
    &\stackrel{(f)}{=} 
    \cf\distf{\z}{\Zeq(\x)}
    + \pr\lipeq\lips\distf{\z}{\Zeq(\x)} 
    + \pr\lipeq\lipr\dists{\x}{\Xeq}
    ,\label{eq:ztp_fix_xt_averaged_case_stoch}
  \end{align}
  where in $(a)$ we use the triangle inequality, in $(b)$ we use the fact that $\bzxz \in \arg\inf_{y \in \Zeq(\x)}\distf{\fast(\x,\z,\rand)}{y}$, the definition of $\distf{\fast(\x,\z,\rand)}{\Zeq(\x)}$, and the $\lipeq$-Lipschitz continuity bound~\eqref{eq:lipschitz_eq_stoch} (cf. Assumption~\ref{ass:set_fix_parametrized_stoch}), in $(c)$ we apply the $\cf$-contractivity of $\fast$ (cf. Assumption~\ref{ass:fast_paracontractive_stoch}) and use the slow evolution~\eqref{eq:SP_system_stochastic_slow}, in $(d)$ we choose $\bzx \in \arg\inf_{y \in \Zeq(\x)}\distf{\z}{y}$, $\bxx \in \arg\inf_{y \in \Xeq}\dists{\x}{y}$, add and subtract within the last norm the term $\slow(\x,\bzx,\rand) = \red(\x,\rand)$, apply the triangle inequality and use the fact that $\red(\bxx,\rand) = 0$ by construction, in $(e)$ we use the Lipschitz continuity of $\slow$ and $\red$ (cf. Assumption~\ref{ass:lip_stoch}), while in $(f)$ we use the fact that $\bzx \in \arg\inf_{y \in \Zeq(\x)}\distf{\z}{y}$ and $\bxx \in \arg\inf_{y \in \Xeq}\dists{\x}{y}$ by definition.
  Hence, by collecting~\eqref{eq:xtp_Xeq_stoch} and~\eqref{eq:ztp_fix_xt_averaged_case_stoch}, we can write 
  \begin{align}
    \begin{bmatrix}
      \E[\dists{\xp}{\Xeq}]
      \\
      \E[\distf{\zp}{\Zeq(\xp)}]
    \end{bmatrix} 
    \leq 
                  M(\pr)
    \begin{bmatrix}
      \dists{\x}{\Xeq}
      \\
      \distf{\z}{\Zeq(\x)}
    \end{bmatrix},\label{eq:norm_evolution_averaged_case_stoch}
  \end{align}
  where $M(\pr)$ is defined in~\eqref{eq:M_pr}, although now the parameters in $M(\pr)$ have different meanings from those in~\eqref{eq:norm_evolution_averaged_case}. %
  Therefore, we can repeat the same arguments to conclude that for all $\pr \in (0,\min\{\bar{\pr},1\})$ (see the definition of $\bar{\pr}$ in~\eqref{eq:bar_pr_stoch}), the matrix $M(\pr)$ is Schur and the proof follows.

  \subsection{Proof of Corollary~\ref{cor:linear_convergence_stoch}}
  \label{sec:proof_linear_convergence_stoch}

   The proof follows the same steps as the proof of Theorem~\ref{th:linear_convergence_without_contraction_stoch} (see Appendix~\ref{sec:proof_th_linear_convergence_without_contraction_stoch}) with the only difference that we consider the distance $\dists{\xp}{\xstar}$ rather than the generic one $\dists{\xp}{\Xeq}$.

\end{document}